\documentclass[11pt,english,british,refpage,intoc,bibliography=totoc,index=totoc,BCOR=7.5mm,captions=tableheading]{extarticle}
\usepackage{mathptmx}
\usepackage[T1]{fontenc}
\usepackage[latin9]{inputenc}
\usepackage[a4paper]{geometry}
\usepackage{color}
\usepackage{babel}
\usepackage{amsmath}
\usepackage{amsthm}
\usepackage{amssymb}
\usepackage{wasysym}
\usepackage[numbers]{natbib}
\usepackage[unicode=true,pdfusetitle,
 bookmarks=true,bookmarksnumbered=true,bookmarksopen=false,
 breaklinks=false,pdfborder={0 0 0},pdfborderstyle={},backref=false,colorlinks=true]
 {hyperref}
\hypersetup{
 linkcolor=black, citecolor=blue, urlcolor=black, filecolor=blue, pdfpagelayout=OneColumn, pdfnewwindow=true, pdfstartview=XYZ, plainpages=false}

\makeatletter
\numberwithin{equation}{section}
\numberwithin{figure}{section}
\theoremstyle{plain}
\newtheorem{thm}{\protect\theoremname}[section]
\theoremstyle{plain}
\newtheorem{lem}[thm]{\protect\lemmaname}

\@ifundefined{date}{}{\date{}}
\usepackage{caption}
\usepackage[nottoc]{tocbibind}
\allowdisplaybreaks[4]
\usepackage{dsfont}
\DeclareMathAlphabet{\mathcal}{OMS}{cmsy}{m}{n}

\makeatother

\addto\captionsbritish{\renewcommand{\lemmaname}{Lemma}}
\addto\captionsbritish{\renewcommand{\theoremname}{Theorem}}
\addto\captionsenglish{\renewcommand{\lemmaname}{Lemma}}
\addto\captionsenglish{\renewcommand{\theoremname}{Theorem}}
\providecommand{\lemmaname}{Lemma}
\providecommand{\theoremname}{Theorem}

\begin{document}
\global\long\def\Sgm{\boldsymbol{\Sigma}}%
\global\long\def\W{\boldsymbol{W}}%
\global\long\def\H{\boldsymbol{H}}%
\global\long\def\P{\mathbb{P}}%
\global\long\def\Q{\mathbb{Q}}%
\newcommand{\dd}{{\rm d\hspace{0.1mm}}}
\newcommand{\R}{{\mathbb{R}}}

\title{New Approach for Vorticity Estimates of Solutions of the Navier-Stokes Equations}
\author{Gui-Qiang G. Chen\thanks{Mathematical Institute, University of Oxford, Oxford OX2 6GG. Email:
\protect\href{mailto:chengq@maths.ox.ac.uk}{chengq@maths.ox.ac.uk}} \,\,\,
and \,\,\,  Zhongmin Qian\thanks{Mathematical Institute, University of Oxford, Oxford OX2 6GG, and OSCAR, Suzhou, China. Email:
\protect\href{mailto:qianz@maths.ox.ac.uk}{qianz@maths.ox.ac.uk}}}
\maketitle

\begin{abstract}
We develop a new approach for regularity estimates, especially vorticity estimates, of solutions of
the three-dimensional Navier-Stokes equations with periodic initial data,
by exploiting carefully formulated linearized vorticity equations.
An appealing feature of the linearized vorticity equations
is the inheritance of the divergence-free
property of solutions, so that it can intrinsically be employed to construct and estimate
solutions of the Navier-Stokes equations.
New regularity estimates of strong solutions of the three-dimensional Navier-Stokes equations
are obtained by deriving new explicit \emph{a priori} estimates for
the heat kernel  ({\it i.e.}, the fundamental solution)
of the corresponding heterogeneous drift-diffusion operator.
These new \emph{a priori} estimates
are derived
by using various functional integral representations of the heat kernel
in terms of the associated diffusion processes and their conditional laws,
including a Bismut-type formula for the gradient of the heat kernel.
Then the \emph{a priori} estimates of solutions of the linearized vorticity equations
are established by employing a Feynman-Kac-type formula. The existence
of strong solutions and their regularity estimates up to a time proportional
to the reciprocal of the square of the maximum initial vorticity are
established.
All the estimates established in this paper contain known
constants that can be explicitly computed.

\selectlanguage{english}%
\medskip


\end{abstract}

\section{Introduction}
We are concerned with the quantitative regularity estimates of
solutions of the
Navier-Stokes equations in $\R^3$.
In this paper, we consider the Cauchy problem for the Navier-Stokes equations for $x\in\mathbb{R}^{3}$ and $t\geq0$:
\begin{equation}\label{t-02}
\begin{cases}
 \partial_t u + u\cdot\nabla u+\nabla P=\nu\Delta u,\\[1mm]
\nabla\cdot u=0, 
\end{cases}
\end{equation}
subject to the periodic initial condition:
\begin{equation}\label{ID}
u|_{t=0} =u_{0},
\end{equation}
satisfying that $u_0(x+L\mathbf{k})=u_0(x)$ for all $\mathbf{k}\in\mathbb{Z}^{3}$
and $\nabla\cdot u_{0}=0$,
where $\nu>0$ is the kinematic constant, and $L>0$ is the period of the initial data.
Although periodic flows are special in nature, the periodic solutions
can be considered as ideal solutions for turbulent motions
away from their physical boundaries, or as models for homogeneous turbulent flows.

The Cauchy
problem \eqref{t-02}--\eqref{ID} for
the Navier-Stokes equations \eqref{t-02}
with periodic
initial data \eqref{ID}
seeks for a velocity
vector field $u(x,t)$ and a scalar pressure $P(x,t)$ that are periodic
functions with period $L$ satisfying the initial condition that $u(x,0)=u_{0}(x)$,
where $u_{0}(x)$ is the periodic initial velocity in \eqref{ID}, which is divergence-free, {\it i.e.},
$\nabla\cdot u_{0}=0$.

The mathematical study of global solutions of the Navier-Stokes equations \eqref{t-02}
was initiated
in the work of Leray \citep{Leray1933,Leray1934} and Hopf \citep{Hopf1941,Hopf1950},
in which global weak solutions
were constructed and investigated.
Since then, many properties and features of
the solutions of the Navier-Stokes equations \eqref{t-02}
have been understood;
see \citep{Galdi1994a,Galdi1994b,Ladyzenskaja1961,Sohr2001,Temam1977,von Wahl1985}
and the references cited therein.
The Navier-Stokes equations \eqref{t-02}
have been
studied
by using various methods; the mathematical analysis
of \eqref{t-02}
has been mainly
based on several functional analysis methods and on the results on
certain functional spaces such as the Sobolev spaces and the Besov spaces.
Great progress has been made in the past decades, the existences of local and global
solutions of the Navier-Stokes equations \eqref{t-02}
have been studied ({\it cf.}
\citep{FabesJonesRiviere1972,FujitaKato1964,Kato1984,LinLei2011,Serrin1962,Serrin1963},
besides the references cited therein and above).
The partial regularity of the weak solutions of the Navier-Stokes equations \eqref{t-02}
has  also
received intensive study; see \citep{CKN1982,Lin1998,Scheffer1976,Scheffer1977,Scheffer1978,Scheffer1980}
and the references cited therein.
However,
the quantitative regularity of solutions, the global existence of strong solutions,
and the uniqueness of weak solutions of the three-dimensional (3-D) Navier-Stokes equations remain to be the major
open problems in Mathematics.

\subsection{Main theorem}
Assume that the periodic initial data function $u_{0}$ is smooth and divergence-free.
Then any strong solution $u(x,t)$ of the Cauchy problem \eqref{t-02}--\eqref{ID}
must have a constant mean velocity, so that
it is assumed without loss of generality that
\begin{equation}\label{1.4a}
\varint_{[0,L]^{3}}u(x,t)\,\textrm{d}x=0,
\end{equation}
due to the Galilean invariance of system \eqref{t-02}.
Denote the vorticity:
\begin{equation}\label{1.5a}
\omega(x,t)=\nabla\wedge u(x,t)
\end{equation}
which is the curl of the velocity.
Then $\omega_{0}=\nabla\wedge u_{0}$ is the initial vorticity.

In this paper, we develop a new approach, {\it i.e.}, an iteration scheme,
to construct and estimate strong solutions of the Cauchy problem \eqref{t-02}--\eqref{ID}
more sharply than the existing results,
by developing an array of useful mathematical tools in Analysis.
The main results for the quantitative estimates of strong solutions can be stated in the following theorem.

\begin{thm}[Main Theorem]\label{thm:main1}
There are two universal constants $C_{1}>0$ and
$C_{2}>0$ such that there exists a unique strong solution $u(x,t)$ of the Cauchy problem
\eqref{t-02}--\eqref{ID} for all $t\in [0, T_{0}]$ with
\begin{equation}
T_{0}=C_{1}\nu^{2}L^{-4}\left\Vert \omega_{0}\right\Vert _{\infty}^{-2},\label{eq:T-zero}
\end{equation}
so that the following estimates for $u(x,t)$ hold{\rm :} For all $0< t\leq T_{0}$ and $x\in\mathbb{R}^{3}$,
\begin{align}
&|u(x,t)|\leq C_{2}L\left\Vert \omega_{0}\right\Vert _{\infty},\qquad
 |\nabla u(x,t)|\le \frac{C_2L}{\sqrt{\nu t}}\left\Vert \omega_{0}\right\Vert_{\infty},\label{eq:c-01e}\\
&|\omega(x,t)|\leq C_{2}\left\Vert \omega_{0}\right\Vert _{\infty},
\qquad |\nabla\omega(x,t)|\leq\frac{C_{2}}{\sqrt{\nu t}}\left\Vert \omega_{0}\right\Vert _{\infty}
 \label{eq:C-02e}
\end{align}
where
$\left\Vert \omega_{0}\right\Vert _{\infty}$
is the $L^{\infty}$-norm of the initial vorticity $\omega_{0}$.
\end{thm}

The two positive constants $C_{1}$ and $C_{2}$ in Theorem \ref{thm:main1}
are computable, which can be worked out by tracing all the universal constants in the proof.

The Cauchy problem \eqref{t-02}--\eqref{ID}
of the Navier-Stokes equations with periodic initial data $u_0(x)$
has been studied traditionally in the Fourier space, which is particularly
the case in turbulence literature;
see, for example, \citep{Batchelor1953,ConstantinFoias1988,Foiasetal2001}.
Our approach is necessary to departure from the well-known methods.
The quantitative regularity estimates, Theorem \ref{thm:main1}, will be proved by
developing a new approach via
an array of mathematical tools from the theory of partial
differential equations, stochastic analysis, and the Hodge theory on
the torus.

\subsection{New approach -- An iteration scheme for the construction and estimates of strong solutions}

We now describe briefly the new approach -- an iteration scheme -- developed in this paper to prove Theorem
\ref{thm:main1}; see \S 6--\S 7 for details.

First of all, by using the dimensionless scaling,
$$
U(x,t)=\frac{L}{2\nu}u(Lx,\frac{L^2}{2\nu} t)
$$
has period $1$ and solves the Navier-Stokes equations \eqref{t-02}
with $\nu=\frac{1}{2}$.
Thus, without loss of generality, we will assume that the viscosity
constant $\nu=\frac{1}{2}$ and $L=1$ in what follows.
Furthermore, from now on, by a (time-dependent) periodic tensor field $f(x,t)$
on $\mathbb{R}^{3}$,
or equivalently by saying that a tensor field $f(x,t)$ is periodic,
we mean that $f(x,t)$ is a tensor field on $\mathbb{R}^{3}$ depending
on the time parameter $t$ and satisfies that $f(x+\bold{k},t)=f(x,t)$
for all $x\in\mathbb{R}^{3}$, $t\geq0$, and $\bold{k}\in\mathbb{Z}^{3}$.

\smallskip
Our new iteration scheme for the construction of strong solutions of
the Cauchy problem \eqref{t-02}--\eqref{ID}
is based on the vorticity equation for $\omega=\nabla\wedge u$:
\begin{equation}
\partial_t\omega+(u\cdot\nabla)\omega-A(u)\omega-\frac{1}{2}\Delta\omega=0,\label{vort-p1-1}
\end{equation}
where $A(u)$ is the total derivative of $u$, a tensor field with
components $A(u)_{j}^{i}=\partial_{x^{j}}u^{i}$.
Although
there are several formulations of the vorticity equations, one of our main observations
is that this version of formulation serves our aims particularly well.

Suppose that $b(x,t)$ is a periodic, smooth, and
divergence-free vector field such that $b(x,0)=u_{0}(x)$.
Then we define a vector field $w(x,t)$, which should be a candidate of the
vorticity (while $w$ is not in general the vorticity of $b$),
by solving the Cauchy problem of the following linear parabolic equations:
\begin{equation}\label{eq:9-06-1}
\begin{cases}
\partial_t w+(b\cdot\nabla)w-A(b)w-\frac{1}{2}\Delta w=0,\\
w(\cdot,0)=\omega_{0},
\end{cases}
\end{equation}
where $A(b)=(A(b)_{j}^{i})=(\partial_{x^j}b^{i})$ is the total
derivative of $b$.
The unique solution $w(x,t)$ has two properties that are important
to our approach:

\begin{enumerate}
\item[\rm (i)] It can be shown that $\nabla\cdot w=0$ (divergence-free again);
this is the key property that makes the linear parabolic equations (\ref{eq:9-06-1})
appealing and workable to our task.

\item[\rm (ii)]
$\varint_{[0,1]^{3}}w(x,t)\,\textrm{d}x=0$ for all $t>0$,
which is satisfied when $t=0$.
\end{enumerate}
With these,
we define the candidate $v(x,t)$ for the velocity by solving
the Poisson equation:
$$
\begin{cases}
\Delta v=-\nabla\wedge w,\\
\varint_{[0,1]^{3}}v(x,t)\,\textrm{d}x=0 \qquad \mbox{for any $t>0$}.
\end{cases}
$$
Then $w=\nabla\wedge v$, according to the Hodge theory on the torus.

In this way, we
construct a mapping $V$ that sends $b(x,t)$
to $v(x,t)$. That is, the iteration for obtaining a strong solution is defined
as
\begin{equation}\label{1.11a}
u^{(n)}=V(u^{(n-1)})\qquad\,\, \mbox{for $n=1,2,\ldots$},
\end{equation}
with the initial
iteration defined by $u^{(0)}(x,t)=u_{0}(x)$ for all $x$ and $t\geq0$.

\smallskip
Let us point out that the computational schemes for simulations of turbulent
flows based on
different formulations
of the vorticity equations have been very fruitful in
the past, {\it cf.} \citep{Chorin 1973,Majda and Bertozzi 2002} for an
overview. Our analysis below shows that the iteration scheme developed in this paper
does converge to the strong solution with inherent gradient estimates
of the iteration solutions uniformly.
This shows that the iteration scheme should also be useful for developing numerical algorithms
to compute
turbulent solutions.

\subsection{Assumptions and notations}

In order to describe the technical aspects in our study, we introduce
a few notations and assumptions that will be used throughout the
paper. By a function or a tensor field we mean a function 
defined on $\mathbb{R}^{d}$ or a tensor field on the torus
$\mathbb{T}^{d}=\mathbb{R}^{d}/\mathbb{Z}^{d}$; in the
latter case, it is identified with a periodic field on $\mathbb{R}^{d}$
with period $1$ in each coordinate variable.

Suppose that $f(x,t)$ for $x\in\mathbb{R}^{d}$
is a tensor field depending on a time parameter $t\geq0$.
Then we assume that $f$ is Borel measurable on $\mathbb{R}^{d}\times[0,\infty)$.
For $I\subset[0,\infty)$, the $L^{\infty}$-norm of $f$
over $\mathbb{R}^{d}\times I$ is defined to be
\[
\left\Vert f\right\Vert _{L^{\infty}(I)}=\sup_{(x,t)\in\mathbb{R}^{d}\times I}\left|f(x,t)\right|,
\]
and for the case when $I=[0,\infty)$, the previous norm is simply
denoted by $\left\Vert f\right\Vert _{\infty}$.

If $0\leq\tau<T$,
the parabolic $L^{\infty}$-norm of $f$ over an interval $[\tau,T]$
will play an important role, which is defined by
\begin{equation}
\left\Vert f\right\Vert _{\tau\rightarrow T}
=\sup_{(x,t)\in\mathbb{R}^{d}\times[\tau,T]}\left|\sqrt{t-\tau}f(x,t)\right|
=\left\Vert \sqrt{\cdot-\tau}f\right\Vert _{L^{\infty}([\tau,T])}.\label{eq:p-norm}
\end{equation}
In the case when $T=\infty$, the interval $[\tau,T]$ is replaced
by $[\tau,\infty)$. From the definition, it is clear that
\begin{equation}
\left\Vert f\right\Vert _{\tau\rightarrow T}
\leq\sqrt{T-\tau}\left\Vert f\right\Vert _{L^{\infty}([\tau,T])}
\leq\sqrt{T-\tau}\left\Vert f\right\Vert _{\infty}.\label{eq:1-01}
\end{equation}

Throughout the paper, the probability density function (PDF) of a normal
random variable with mean zero and variance $t>0$ is denoted by
\begin{equation}
G_{t}(x)=\frac{1}{(2\pi t)^{d/2}}\exp(-\frac{|x|^{2}}{2t})
\qquad \mbox{for $x\in\mathbb{R}^{d}$}. \label{re-G1}
\end{equation}
Notice that
$G_{t-\tau}(y-x)$ is the
fundamental solution of the heat operator $\partial_t-\frac{1}{2}\Delta$
in the Euclidean space $\mathbb{R}^{d}$.

As a convention, the Laplacian $\Delta$ and
the gradient
$\nabla$ (in particular, the divergence operation $\nabla\cdot$ and
the curl $\nabla\wedge$), when operating on time-dependent tensor
fields on $\mathbb{R}^{d}$, apply to space variable $x$ only.

If $b(x,t)$ is a time-dependent vector field on $\mathbb{R}^{d}$
for $t\geq0$,
the heterogeneous drift-diffusion
differential operator of second
order:
\begin{equation}
L_{b(x,t)}=\frac{1}{2}\Delta+b(x,t)\cdot\nabla\label{eq:form 2}
\end{equation}
will play an important role in our study. When no confusion arises,
$L_{b(x,t)}$ is denoted simply by $L_{b}$.
Among the technical assumptions on $b(x,t)$, the most essential one
is the assumption
that $b(x,t)$ is \emph{solenoidal} ({\it i.e.}, \emph{divergence-free}).
That is, for every $t$, the divergence $\nabla\cdot b(\cdot,t)$ vanishes
identically in the distributional sense.
Under this assumption, the formal
adjoint operator:
$$
L_{b}^{\star}=L_{-b},
$$
which is again an elliptic
operator of the same type.
The second assumption is technical for
the construction of probabilistic structures. For simplicity, we assume
that $b(x,t)$ is Borel measurable and bounded over any finite
interval, {\it i.e.}, $\left\Vert b\right\Vert _{L^{\infty}(\mathbb{R}^d\times [0,T])}<\infty$
for every $T>0$. The probability density function of the $L_{b}$-diffusion
({\it i.e.}, the fundamental solution or heat kernel to the parabolic operator
$L_{b(x,t)}^{\star}+\partial_t$; see \S 2) is
denoted by
$$
p_{b}(\tau,x,t,y) \qquad \mbox{for $t>\tau\geq0$ and $x,y\in\mathbb{R}^{d}$}.
$$

Throughout the paper, universal constants (the constants depending only
on the dimension $d$, or some parameters $\beta$, $\gamma$, {\it etc.}
introduced in proofs) are denoted by $C_{1}$, $C_{2}$, {\it etc}.
which may be different at each occurrence.

\subsection{\emph{A priori} estimates of solutions of the parabolic equations}

In order to prove Theorem \ref{thm:main1}, the main effort is to
derive precise \emph{a priori} estimates of solutions of the Cauchy problem \eqref{eq:9-06-1} for the parabolic equations.
To achieve this, there are two tasks to be
carried out.

\medskip
1. {\it The main task is to derive explicit \emph{a priori} estimates for
the fundamental solution {\rm (}or called the heat kernel{\rm )} of the parabolic
operator $\partial_{t}-L_{b}$ and its gradient
in terms of the bound of $b$ and the parabolic norm of $\nabla b$,
when $b(x,t)$ is a bounded, divergence-free, and smooth vector field
on $\mathbb{R}^{d}$}.

\smallskip
Although the regularity theory for linear parabolic
equations has been well
established (cf. \citep{Friedman 1964,Ladyzenskaja Solonnikov and Uralceva 1968,Stroock 2008}),
our \emph{a priori} estimates contain the universal constants depending
only on the dimension $d$ and a parameter $\beta>1$
fixed in 
our estimates.
In addition to its explicit form, for a divergence-free vector field $b(x,t)$,
the gradient estimate for the heat kernel
associated with the parabolic operator $\partial_t-L_{b}$,
depends only on the bound of $b$ and its first-order derivative $\nabla b$.

\begin{thm}\label{thm:b-est}
Let $b(x,t)$ be a smooth, divergence-free, and bounded time-dependent
vector field on $\mathbb{R}^{d}$. Then, for every $\beta>1$,
there are constants $C_{1}$ and $C_{2}$ depending only on $\beta$
and dimension $d$ such that
\begin{align}
&p_{b}(\tau,x,t,y)\leq C_{1}\textrm{e}^{C_{2}(t-\tau)
  \left\Vert b\right\Vert _{\infty}^{2}}G_{\beta(t-\tau)}(y-x),\label{eq:p-est1}\\
&\left|\nabla_{y}p_{b}(\tau,x,t,y)\right|
   \leq\frac{C_{1}}{\sqrt{t-\tau}}\textrm{e}^{C_{2}(t-\tau)
   \left\Vert b\right\Vert _{\infty}^{2}
   +\frac{1}{2}\sqrt{t-\tau}
   \left\Vert \nabla b\right\Vert _{\tau\rightarrow t}}G_{\beta(t-\tau)}(y-x),\label{eq:grad-e2}
\end{align}
for all $t>\tau$ and $x,y\in\mathbb{R}^{d}$.
\end{thm}

These estimates are quite delicate
to derive: They are obtained
by introducing substantial tools from stochastic analysis, mainly
various functional integral representations for the fundamental solutions
({\it cf.} \citep{QianZhengSPA2004,QianSuliZhang2022}) and a new kind of
Bismut's formulas (cf. \citep{Bismut1984,Bismut-Lebeau2008}), together
with careful and explicit computations. Estimates (\ref{eq:p-est1})--(\ref{eq:grad-e2})
will be proved in \S 4 and \S 5, respectively.

\medskip
2. {\it The second task in our study is to prove the explicit \emph{a priori}
estimates to the iteration $u^{(n)}=V(u^{(n-1)})$ in \eqref{1.11a},
or equivalently, to derive the {\it a priori} estimates of the solutions
of the Cauchy problem \eqref{eq:9-06-1} for the linear parabolic
equations that define the nonlinear mapping $V$, namely the solution
of the Cauchy problem for the parabolic equations:
\begin{equation}\label{eq:vort-linear-01}
\begin{cases}
(\partial_t-L_{-b})w=A(b)w,\\
w(\cdot,0)=\omega_{0},
\end{cases}
\end{equation}
where, as before, $L_{-b}$ denotes the time-dependent elliptic operator
$\frac{1}{2}\Delta-b\cdot\nabla$}.

\smallskip
The crucial observation is that
the term on the right-hand side, $A(b)w$, is a linear zero-order term,
which differs fundamentally from the linearized Navier-Stokes equations.
This crucial difference allows us to apply the Feynman-Kac-type formula,
obtained in this context in
\citep{QianSuliZhang2022},
to derive the necessary explicit \emph{a priori} estimates.

The \emph{a priori} estimates and technical tools are worked out for
a general dimension $d$ and a general vector field $b(x,t)$ that
is divergence-free on $\mathbb{R}^{d}$. Therefore, they have
independent interests and are likely useful for both treating the Navier-Stokes
equations with other boundary conditions and dealing with other
linear/nonlinear PDEs.

\subsection{Organisation of the paper}

In \S 2, several probabilistic structures associated with a time-dependent
vector field $b(x,t)$ are reviewed, and then a functional integration
representation formula for the heat kernel of $\frac{1}{2}\Delta-b\cdot\nabla$
and a Bismut-type formula for the gradient of the heat kernel are established,
which are the tools for deriving the \emph{a priori} estimates
in Theorem \ref{thm:b-est} we need.
In \S 3, several technical potential estimates are established, which will
be used in the proof of Theorem \ref{thm:b-est} that will be carried out
in \S 4--\S 5.
In \S 6,
the linearized
vorticity equations are carefully analyzed, and the main regularity results for the strong
solutions of the Cauchy problem \eqref{t-02}--\eqref{ID}
with periodic initial data \eqref{ID}
for the Navier-Stokes equations \eqref{t-02}
in $\R^3$
will be proved in \S 7.

\section{Probabilistic Tools for the New Approach}

In this section, we first introduce several probabilistic structures associated
with a time-dependent vector field $b(x,t)\in \mathbb{R}^{d}$
that is divergence-free (in the distributional sense), bounded,
and Borel measurable. Then we establish a functional integration representation
formula and a Bismut-type formula for the heat kernel of the corresponding
heterogeneous drift-diffusion operator.

\subsection{Fundamental solutions and diffusions}

Let $\varGamma_{b}(x,t,\xi,\tau)$, for $0\leq\tau<t$ and $\xi,x\in\mathbb{R}^{d}$,
denote the fundamental solution of the forward parabolic operator
$L_{b}-\partial_t$, and let $\varGamma_{b}^{\star}(x,t,\xi,\tau)$,
for $0\leq t<\tau$ and $x,\xi\in\mathbb{R}^{d}$, denote the fundamental solution
of the backward parabolic equation $L_{b}^{\star}+\partial_t$;
see \citep{Friedman 1964} for their definitions and basic constructions.
Since $b(x,t)$ is bounded, $\varGamma_{b}(x,t,\xi,\tau)$ and $\varGamma_{b}^{\star}(x,t,\xi,\tau)$
exist and unique, and
\begin{equation}\label{eq:fund-gama1}
\varGamma_{b}(x,t,\xi,\tau)=\varGamma_{b}^{\star}(\xi,\tau,x,t)
\qquad \mbox{for any $t>\tau\geq0$ and $x,\xi\in\mathbb{R}^{d}$}.
\end{equation}
Moreover, $\varGamma_{b}(x,t,\xi,\tau)$
is positive and continuous in $\tau<t$ and $x,\xi\in\mathbb{R}^{d}$,
and $\varGamma_{b}^{\star}(x,t,\xi,\tau)$ is positive and continuous
in $t<\tau$ and $x,\xi\in\mathbb{R}^{d}$.

Let $p_{b}(\tau,\xi,t,y)$, for $0\leq\tau<t$ and $\xi,y\in\mathbb{R}^{d}$,
denote the transition probability density function of the $L_{b}$-diffusion
({\it cf}. \citep{Ikeda and Watanabe 1989,Stroock and Varadhan 1979,Stroock 2008}).
Since $\nabla\cdot b=0$,
\begin{equation}
p_{b}(\tau,\xi,t,y)=\varGamma_{-b}^{\star}(\xi,\tau,y,t)=\varGamma_{-b}(y,t,\xi,\tau)\label{def-r1}
\qquad\mbox{for $t>\tau\geq0$ and any $\xi,y\in\mathbb{R}^{d}$}.
\end{equation}
 For $T>0$,
$b^{T}(x,t)$ denotes a bounded divergence-free vector field such
that $b^{T}(x,t)$ coincides with $b(x,T-t)$ for all $0\leq t\leq T$
and $x\in\mathbb{R}^{d}$. Then
\begin{align}
&p_{b^{T}}(T-t,y,T-\tau,\xi)=\varGamma_{b}(y,t,\xi,\tau),\label{eq:diff-backT1}\\
&p_{b}(\tau,\xi,t,y)=p_{-b^{T}}(T-t,y,T-\tau,\xi),\label{eq:rev-01}
\end{align}
for all $0\leq\tau<t\leq T$ and $\xi,y\in\mathbb{R}^{d}$.

Let $\varOmega=C([0,\infty),\mathbb{R}^{d})$ be the space of continuous
paths $\varphi:[0,\infty)\rightarrow\mathbb{R}^{d}$, equipped with
the natural filtration denoted by $\mathcal{F}_{t}^{0}$ for $t\geq0$.
For $\tau\geq0$ and $\xi\in\mathbb{R}^{d}$, there is a unique probability
measure $\mathbb{P}_{b}^{\tau,\xi}$ on $\varOmega$ of all continuous
paths such that
\[
\mathbb{P}_{b}^{\tau,\xi}\left[\varphi\in\varOmega\,: \,\varphi(t)=\xi
\,\textrm{ for all }0\leq t\leq\tau\right]=1,
\] 
and the marginal distribution for any finite partition $\tau=t_{0}<t_{1}<\cdots<t_{k}$:
\[
\mathbb{P}_{b}^{\tau,\xi}\left[\varphi(t_{1})\in \dd x_{1},\cdots,\varphi(t_{k})\in \dd x_{k}\right]
\]
is given by
\[
p_{b}(\tau,\xi,t_{1},x_{1})p_{b}(t_{1},x_{1},t_{2},x_{2})\cdots p_{b}(t_{k-1},x_{k-1},t_{k},x_{k})\dd x_{1}\cdots \dd x_{k}.
\]

The collection $\mathbb{P}_{b}^{\tau,\xi}$, for $\tau\geq0$ and $\xi\in\mathbb{R}^{d}$,
is called the diffusion with infinitesimal generator $L_{b}$. The
construction of $L_{b}$-diffusion, or equivalently of $p_{b}(s,x,t,y)$,
can be based on the Cauchy problem for the stochastic differential equation (SDE):
\begin{equation}\label{eq:sde-m1}
\begin{cases}
\textrm{d}X=b(X,t)\textrm{d}t+\textrm{d}B,\\
X_{\tau}=\xi,
\end{cases}
\end{equation}
where $B=(B^{1},\cdots,B^{d})$ is a Brownian motion on a probability
space $(\varOmega,\mathcal{F},\mathbb{P})$
({\it cf.} \citep{Ikeda and Watanabe 1989, Stroock and Varadhan 1979} for the details).
If $b(x,t)$
is jointly continuous and global Lipschitz continuous in $x$ (uniformly
in $t$), then the Cauchy problem (\ref{eq:sde-m1}) has a unique strong solution,
whose distribution gives rise to the probability measure $\mathbb{P}_{b}^{\tau,\xi}$.

\subsection{Diffusion bridges and a Feynman-Kac-type formula}

Let $\tau\geq0$, $T>\tau$, and $\xi,\eta\in\mathbb{R}^{d}$ be fixed.
The conditional law $\mathbb{P}_{b}^{\xi,\tau\rightarrow\eta,T}$,
also called the pinned measure or $L_{b}$-diffusion bridge measure,
is formally defined to be $\mathbb{P}_{b}^{\tau,\xi}\left[\,\,\cdot\,\,|\,w(T)=\eta\right]$,
which is again Markovian. According to (14.1) in \citep{Dellacherie and Meyer Volume D},
$\mathbb{P}_{b}^{\xi,\tau\rightarrow\eta,T}$ is the unique probability
measure on $\varOmega$ with time non-homogeneous transition probability
density function:
\begin{equation}\label{cond-01}
q_{b}(s,x,t,y)=\frac{p_{b}(s,x,t,y)p_{b}(t,y,T,\eta)}{p_{b}(s,x,T,\eta)}
\qquad\,\, \mbox{for $\tau<s<t<T$ and $x,y\in\mathbb{R}^{d}$}.
\end{equation}
It can be shown ({\it cf.} \citep{QianZhengSPA2004}) that
\begin{equation}\label{eq:den-01}
\left.\frac{\textrm{d}\mathbb{P}_{b}^{\xi,\tau\rightarrow\eta,T}}{\textrm{d}\mathbb{P}_{b}^{\xi,\tau}}
\right|_{\mathcal{F}_{s}^{0}}
=\frac{p_{b}(t,\varphi(t),T,\eta)}{p_{b}(\tau,\xi,T,\eta)}
\qquad\,\,\mbox{for $t\in[\tau,T)$},
\end{equation}
where $\varphi(t)$ denotes the general sample point and the canonical process on $\Omega$.
\begin{thm}
\label{theorem3.3}The pinned measures satisfy the following duality
relation{\rm :}
\[
\mathbb{P}_{b}^{\eta,0\rightarrow\zeta,T}=\mathbb{P}_{-b^{T}}^{\zeta,0\rightarrow\eta,T}\circ\tau_{T},
\]
where $\tau_{T}$ is the time reversal operator at $T$, that is,
$\tau_{T}:\varOmega\rightarrow\varOmega$ which sends $w$ to $\tau_{T}w(t)=w(T-t)$
for $t\in[0,T]$.
\end{thm}

Therefore, if $\{X_{t}\}$ is an $L_{b}$-diffusion and $\{Y_{t}\}$ is
an $L_{-b^{T}}$-diffusion on a probability space $(\varOmega,\mathcal{F},\mathbb{P})$,
then
\[
\mathbb{P}\left[F(X_{\cdot})\,|\,X_{0}=\eta,X_{T}=\zeta\right]
=\mathbb{P}\left[F(Y_{T-\cdot})\,|\,Y_{0}=\zeta,Y_{T}=\eta\right]
\]
for any bounded or positive Borel measurable function $F$ on $\varOmega$.

The following is a Feynman-Kac-type formula that was
established in \citep{QianSuliZhang2022} in this version:
\begin{thm}\label{thm:F-K-01}
Suppose that $f(x,t)$ is a strong solution to the parabolic
equations{\rm :}
\begin{equation}
\big(\partial_t-L_{-b(x,t)}\big)f^{i}(x,t)
=A_{j}^{i}(x,t)f^{j}(x,t)
\qquad \mbox{ in $\mathbb{R}^{d}\times[0,\infty)$},\label{eq:S-vot-eq}
\end{equation}
subject to the initial condition{\rm :} $f(x,0)=f_{0}(x)$, and $A_{j}^{i}(x,t)$
are joint continuous in $(x,t)$, Lipschitz continuous in $x$ $($uniformly
in $t$ in any finite interval, where $i,j=1,\ldots,d${\rm )}. Then
\begin{equation}
f(x,t)=\int_{\mathbb{R}^{d}}f_{0}(\xi)p_{b}(0,\xi,t,x)\mathbb{P}_{b}^{\xi,0\rightarrow x,t}\left[Q(0,t)\right]
 \dd\xi,\label{eq:F-K-f01}
\end{equation}
where, for every $t>0$, $Q(s,t,\varphi)=\big(Q_{j}^{i}(s,t,\varphi)\big)$
$($but the sample point $\varphi\in\varOmega$ will be suppressed if
no confusion arises$)$ is the unique solution of the Cauchy problem for the differential
equations{\rm :}
\begin{align*}
\begin{cases}
\frac{\dd}{\dd s}Q_{j}^{i}(s,t,\varphi)=-Q_{k}^{i}(s,t,\varphi)A_{j}^{k}(\varphi(s),s),\\[1mm]
Q_{j}^{i}(t,t,\varphi)=\delta_{j}^{i} \qquad \mbox{for $s\leq t$},
\end{cases}
\end{align*}
for $i,j=1,\ldots,d$, and $\varphi\in\varOmega$.
\end{thm}

\subsection{Functional integral representation for fundamental solutions}

Let $B$ be a standard Brownian motion of dimension $d$ on a probability
space $(\varOmega,\mathcal{F},\mathbb{P})$. For $\tau\geq0$ and
$\xi\in\mathbb{R}^{d}$,
$$
X_{t}^{\tau,\xi}=\begin{cases}
B_{t}-B_{\tau}+\xi \quad &\mbox{for $t\geq\tau$},\\
\xi \quad &\mbox{for $t\leq\tau$}.
\end{cases}
$$
Then $X^{\tau,\xi}$
is a Brownian motion stated at $\xi$ at the initial time $\tau$.
According to Theorem 6.4.2 in \citep{Stroock and Varadhan 1979},
the $L_{b}$-diffusion may also be constructed
by using the Cameron-Martin density:
\begin{equation}\label{eq:U-1}
U_{b}^{\tau,\xi}(t)=\exp\Big[\int_{\tau\wedge t}^{t}b(X_{s}^{\tau,\xi},s)\cdot\textrm{d}B_{s}
-\frac{1}{2}\int_{\tau\wedge t}^{t}|b|^{2}(X_{s}^{\tau,\xi},s)\,\textrm{d}s\Big].
\end{equation}

\begin{thm}
\label{thm:3.1} The following representation holds for all $t>\tau${\rm :}
\begin{equation}\label{eq:rep-05}
p_{b}(\tau,\xi,t,y)=G_{t-\tau}(y-\xi)
+\int_{\tau}^{t}\mathbb{P}\Big[U_{b}^{\tau,\xi}(s)G_{t-s}(y-X_{s}^{\tau,\xi})
b(X_{s}^{\tau,\xi},s)\cdot\frac{y-X_{s}^{\tau,\xi}}{t-s}\Big]\,\textrm{\rm d}s.
\end{equation}
\end{thm}

\begin{proof}
The theorem was established in \citep{QianZhengSPA2004} for the time
homogeneous case, and their method can be developed to be applied to our case.
For the completeness, we outline the proof here.
According to the Cameron-Martin
formula,
\begin{align*}
\int_{\mathbb{R}^{d}}p_{b}(\tau,\xi,t,x)f(x)\,\textrm{d}x
& =\mathbb{P}\big[U_{b}^{\tau,\xi}(t)f(X_{t}^{\tau,\xi})\big]\\
& =\int_{\mathbb{R}^{d}}\mathbb{P}\big[U_{b}^{\tau,\xi}(t)f(x)\big|X_{t}^{\tau,\xi}=x\big]
  \mathbb{P}\big[X_{t}^{\tau,\xi}\in\textrm{d}x\big]\\
& =\int_{\mathbb{R}^{d}}
 \mathbb{P}\big[U_{b}^{\tau,\xi}(t)\big|X_{t}^{\tau,\xi}=x\big]G_{t-\tau}(x-\xi)f(x)\,\textrm{d}x.
\end{align*}
Choose $f(z){\rm d}z=\delta_{x}(\textrm{d}z)$ to obtain
\begin{equation}\label{eq:rep-01}
\frac{p_{b}(\tau,\xi,t,y)}{G_{t-\tau}(y-\xi)}
=\mathbb{P}\big[U_{b}^{\tau,\xi}(t)\big|X_{t}^{\tau,\xi}=y\big].
\end{equation}
Now we notice that both
\[
R(s)=\frac{G_{t-s}(y-X_{s}^{\tau,\xi})}{G_{t-\tau}(y-\xi)}
\]
(see (\ref{eq:den-01})) and $U_{b}^{\tau,\xi}(s)$ are exponential
martingales so that
\begin{align*}
&\textrm{d}R(s)=R(s)\nabla\ln G_{t-s}(y-X_{s}^{\tau,\xi})\cdot\textrm{d}B(s),\qquad R(\tau)=1,\\
&\textrm{d}U_{b}^{\tau,\xi}(s)=U_{b}^{\tau,\xi}(s)b(X_{s}^{\tau,\xi},s)\cdot\textrm{d}B(s),
\qquad\quad\,\,\, U_{b}^{\tau,\xi}(\tau)=1.
\end{align*}
Therefore, integrating by parts, together with (\ref{eq:den-01}), yields
\begin{align*}
\frac{p_{b}(\tau,\xi,t,y)}{G_{t-\tau}(y-\xi)}
& =\lim_{s\uparrow t}\mathbb{P}\big[R(s)U_{b}^{\tau,\xi}(s)\big]\\
 & =1+\mathbb{P}\big[\big\langle R,U_{b}^{\tau,\xi}\big\rangle _{t}\big]\\
 & =1+\frac{1}{G_{t-\tau}(y-\xi)}
 \mathbb{E}\Big[\int_{\tau}^{t}U_{b}^{\tau,\xi}(s)b(X_{s}^{\tau,\xi},s)\cdot\nabla G_{t-s}(y-X_{s}^{\tau,\xi})\,\textrm{d}s\Big],
\end{align*}
which leads to the representation formula in \eqref{eq:rep-05}.
\end{proof}

\subsection{A Bismut-type formula}

From now on, we make a further assumption that $b(x,t)$ is
a time-dependent, bounded, and divergence-free $C^{1}$-vector field.

Using the idea of Bismut ({\it cf}. Chapter 14, \S 14.1 in Bismut-Lebeau \citep{Bismut-Lebeau2008}, and also {\it cf}. Elworthy-Li \citep{Elworthy-Li}),
we first establish a Bismut-type formula for the gradient of the fundamental solution:
\[
\partial_{x^{j}}\ln p_{b}(\tau,\xi,t,x).
\]

Let $\tau\geq0$ and $\xi\in\mathbb{R}^{d}$ be fixed. Bismut's idea
is based on the following observation: Since $b$ is divergence-free, as a function of $t>\tau$
and $x\in\mathbb{R}^{d}$, $p_{b}(\tau,\xi,t,x)$ solves the forward
parabolic equation $(L_{-b}-\partial_t)p_{b}=0$, so that
\begin{equation}\label{eq:q-01}
\big(L_{-b}-\partial_t\big)\ln p_{b}=-\frac{1}{2}\left|\nabla\ln p_{b}\right|^{2}
\qquad \mbox{on $\mathbb{R}^{d}\times(\tau,\infty)$}.
\end{equation}

Let $T>0$. Consider $f(x,t)=\ln p_{b}(\tau,\xi,\tau+T-t,x)$ for
$0\leq t<T$. Then
\begin{equation}
\big(L_{-b}+ \partial_t\big)f=-\frac{1}{2}\left|\nabla f\right|^{2},\label{eq:f-ee}
\end{equation}
where both sides are evaluated at $(\tau,\xi,\tau+T-t,x)$.
Now we consider the following Cauchy problem of the stochastic differential equations:
\begin{equation}\label{eq:Y-sde1}
\begin{cases}
\textrm{d}Y_{t}^{x}=\textrm{d}B_{t}-b(Y_{t}^{x},T+\tau-t)\,\textrm{d}t,\\
Y_{0}^{x}=x,
\end{cases}
\end{equation}
which determines a diffusion process with generator $L_{-b^{T+\tau}}$, where $B$
is a standard $d$-dimensional Brownian motion on a probability space
$(\varOmega,\mathcal{F},\mathbb{P})$, whose filtration is denoted
by $(\mathcal{F}_{t}^{0})_{t\geq0}$.
Then \eqref{eq:f-ee} implies that
\begin{equation}\label{eq:R-d1}
R_{t}^{x}=\frac{p_{b}(\tau,\xi,T+\tau-t,Y_{t}^{x})}{p_{b}(\tau,\xi,T+\tau,x)}
=\textrm{e}^{f(Y_{t}^{x},t)-f(x,0)}\qquad\mbox{for $t\in[0,T)$}
\end{equation}
 is a positive martingale and
\begin{equation}\label{eq:R-sde-01}
\begin{cases}
\textrm{d}R_{t}^{x}=R_{t}^{x}\nabla f(Y_{t}^{x},t)\cdot\textrm{d}B_{t},\\
R_{0}^{x}=1.
\end{cases}
\end{equation}
Finally, we set $Z_{j}^{k}(x,t)=\partial_{x^{j}}Y_{t}^{x,k}$  for $j,k=1,\cdots,d$.

\begin{thm}[Bismut-Type Formula for the Forward Variable]\label{thm:bis-f}
Suppose that $b(x,t)$ is bounded and $C^{1}$ with bounded derivative over any finite interval
so that $\nabla\cdot b=0$. Then
\begin{equation}\label{eq:bis-01}
\partial_{x^j}\ln p_{b}(\tau,\xi,T+\tau,x)
=\mathbb{Q}\Big[\int_{0}^{T}\frac{\rho'(t)}{\rho(T)}Z_{j}^{k}(x,t)\,\textrm{d}B_{t}^{k}\Big],
\end{equation}
where $\mathbb{Q}$ is the probability measure on $(\varOmega,\mathcal{F}_{T}^{0})$
such that
\[
\left.\frac{\textrm{d}\mathbb{Q}}{\textrm{d}\mathbb{P}}\right|_{\mathcal{F}_{t}^{0}}=R_{t}^{x}
\qquad\quad \mbox{for $t\in[0,T)$},
\]
and $\rho(t)$ is any continuous and piecewise differentiable
function with $\rho(0)=0$ and $\rho(t)>0$ for $t>0$.
\end{thm}

\begin{proof}
Since $t\rightarrow R_{t}^{x}$ is a martingale so that
\[
p_{b}(\tau,\xi,T+\tau,x)=\mathbb{P}\left[p_{b}(\tau,\xi,T+\tau-t,Y_{t}^{x})\right],
\]
where $\mathbb{P}$ (similarly for $\mathbb{Q}$) also means taking expectation with respect to $\mathbb{P}$ (resp. $\mathbb{Q}$). By using the previous fact,
\begin{align*}
\partial_{x^{j}}\ln p_{b}(\tau,\xi,T+\tau,x)
& =\frac{\partial_{x^{j}}p_{b}(\tau,\xi,T+\tau,x)}{p_{p}(\tau,\xi,T+\tau,x)}\\
 & =\frac{\mathbb{P}\left[\partial_{x^{j}}p_{b}(\tau,\xi,T+\tau-t,Y_{t}^{x})\right]}{p_{p}(\tau,\xi,T+\tau,x)}\\
 & =\mathbb{P}\Big[\frac{p_{b}(\tau,\xi,T+\tau-t,Y_{t}^{x})}{p_{p}(\tau,\xi,T+\tau,x)}\partial_{x^{j}}\ln p_{b}(\tau,\xi,T+\tau-t,Y_{t}^{x})\Big]\\
 & =\mathbb{P}\big[R_{t}^{x}Z_{j}^{k}(x,t)
 \partial_{x^j}\ln p_{b}(\tau,\xi,T+\tau-t,Y_{t}^{x})\big]
\qquad\,\,\, \mbox{for all $0<t<T$},
\end{align*}
where $Z_{j}^{k}(x,t)=\partial_{x^{j}}Y_{t}^{x,k}$
for $j,k=1,\cdots,d$. Let
\[
M_{t}^{j}=\int_{0}^{t}\rho'(s)Z_{j}^{k}(s)\,\textrm{d}B_{s}^{k}
\qquad\,\,\mbox{for $0<t<T$}.
\]
Observe that
\[
\left\langle R,M^{j}\right\rangle _{t}
=\int_{0}^{t}\rho'(s)Z_{j}^{k}(s)R_{s} \partial_{x^{k}}\ln p_{b}(\tau,\xi,T+\tau-s,Y_{s})\,\textrm{d}s.
\]
Thus we have
\begin{align*}
\partial_{x^{j}}\ln p_{b}(\tau,\xi,T+\tau,x)
& =\frac{1}{\rho(t)}\mathbb{P}\Big[\int_{0}^{T}\rho'(s)Z_{j}^{k}(s)R_{s}
\partial_{x^{k}}\ln p_{b}(\tau,\xi,T+\tau-s,Y_{s})\,\textrm{d}s\Big]\\
 & =\frac{1}{\rho(T)}\mathbb{P}\left[\left\langle R^{x},M^{j}\right\rangle _{T}\right]\\
 & =\frac{1}{\rho(T)}\mathbb{P}\Big[R_{T}^{x}\int_{0}^{T}\rho'(t)Z_{j}^{k}(t)\,\textrm{d}B_{t}^{k}\Big].
\end{align*}
This completes the proof.
\end{proof}

\section{Basic Estimates for the Heat Kernel}
In this section, we establish several estimates for the heat kernel.
In particular, Lemma \ref{lem:4.2} contains the technical estimates needed for the proof of
Theorem \ref{prop:4.3} below.

Recall that $G_{t}(x)=(2\pi t)^{-\frac{d}{2}}\exp(-\frac{|x|^{2}}{2t})$
for $x\in\mathbb{R}^{d}$ and $t>0$. Clearly,
\begin{equation}\label{re-G3}
\nabla_{x}\ln G_{t-s}(y-x)=\frac{y-x}{t-s} \qquad\,\, \mbox{for $t>s\geq0$},
\end{equation}
 which leads to the following equality:
\begin{equation}\label{re-D4}
|\sqrt{t-s}\nabla_{x}G_{t-s}(y-x)|=\frac{|y-x|}{\sqrt{t-s}}G_{t-s}(y-x).
\end{equation}
Now we notice that, for $\beta>0$,
\begin{equation}
G_{t-s}(y-x)=\beta^{\frac{d}{2}}e^{-\frac{\beta-1}{\beta}\frac{|y-x|^{2}}{2(t-s)}}G_{\beta(t-s)}(y-x),\label{re-G6}
\end{equation}
so that, for $\beta>1$,
\begin{align*}
\frac{|y-x|}{\sqrt{t-s}}G_{t-s}(y-x)
& =\beta^{\frac{d}{2}}\frac{|y-x|}{\sqrt{t-s}}e^{-\frac{\beta-1}{\beta}\frac{|y-x|^{2}}{2(t-s)}}G_{\beta(t-s)}(y-x)\\
& =\beta^{\frac{d}{2}}\frac{|y-x|}{\sqrt{t-s}}\frac{1}{1+\frac{\beta-1}{\beta}\frac{|y-x|^{2}}{2(t-s)}+\cdots}G_{\beta(t-s)}(y-x)\\
& \leq\frac{\beta^{\frac{d+1}{2}}}{\sqrt{2(\beta-1)}}G_{\beta(t-s)}(y-x).
\end{align*}
Therefore, using \eqref{re-D4} and the same argument,
we have the following lemma which will be used in our computations later.
\begin{lem}
For the heat kernel in $\mathbb{\mathbb{R}}^{d}$,  the following estimates
hold{\rm :}
\begin{enumerate}
\item[\rm (i)] For $\beta>1$,
\begin{equation}
|\sqrt{t-s}\nabla_{x}G_{t-s}(y-x)|\leq\frac{\beta^{\frac{d+1}{2}}}{\sqrt{2(\beta-1)}}G_{\beta(t-s)}(y-x)\label{model-der-e1}  \qquad\mbox{for all $t>s\geq0$ and $x,y\in\mathbb{R}^{d}$}.
\end{equation}
\item[\rm (ii)]
For $\alpha\geq0$ and $\beta>1$,
\begin{equation}\label{SDE-Gest1}
\Big(\frac{|y-x|}{\sqrt{t-s}}\Big)^{\alpha}G_{t-s}(y-x)
\leq\beta^{\frac{d}{2}}k!\Big(\frac{2\beta}{\beta-1}\Big)^{k}G_{\beta(t-s)}(y-x)
\qquad\mbox{for all $t>s\geq0$ and $x,y\in\mathbb{R}^{d}$},
\end{equation}
where $k=\left[\frac{\alpha}{2}\right]+1$
if $\frac{\alpha}{2}$ is not an integer, and $k=\frac{\alpha}{2}$
otherwise.
\end{enumerate}
\end{lem}

\begin{lem}\label{lem:4.1}
Let $B$ be an $\mathbb{R}^{d}$-Brownian motion
on a probability space $(\varOmega,\mathcal{F},\mathbb{P})$. Consider
\begin{equation}\label{eq:I-a-b}
I_{\alpha,\beta}(\tau,x,s,t,y)
=\mathbb{P}\Big[\Big(\big|\frac{y-X_{s}}{t-s}\big|^{\alpha}G_{t-s}(y-X_{s})\Big)^{\beta}\Big]
\qquad\mbox{for $t>s>\tau$ and $x,y\in\mathbb{R}^{d}$},
\end{equation}
where $\alpha\geq0$,
$\beta>0$ and $X_{s}=B_{s}-B_{\tau}+x$. Then
\begin{equation}
I_{\alpha,\beta}(\tau,x,s,t,y)
=C_{1} t_{1}^{-(\frac{d(\beta-1)}{2}+\alpha\beta)} G_{t_{1}+t_{2}}(y-x)
\int_{\mathbb{R}^{d}}
\Big|\sqrt{\frac{t_{1}t_{2}}{t_{1}+t_{2}}}z+\frac{t_{1}(x-y)}{t_{1}+t_{2}}\Big|^{\alpha\beta}
G_{1}(z)\,{\rm d} z,\label{eq:I-a-b-1}
\end{equation}
where
\begin{equation}\label{eq:tonettwo}
t_{1}=\beta^{-1}(t-s),\quad t_{2}=s-\tau,
\end{equation}
and
\begin{equation}\label{eq:C_1}
C_{1}=
(2\pi)^{-\frac{d(\beta-1)}{2}}\beta^{-(\frac{d}{2}+\alpha)\beta}
\end{equation}
depend only on $\beta>0$ and $d$.
\end{lem}

\begin{proof}
The proof follows from an elementary computation. First notice that
\begin{equation}\label{Normal-prod01}
G_{t_{1}}(z-y)G_{t_{2}}(z-x)
=G_{\frac{t_{1}t_{2}}{t_{1}+t_{2}}}(z-\frac{t_{1}x+t_{2}y}{t_{1}+t_{2}})G_{t_{1}+t_{2}}(y-x)
\qquad\mbox{for $t_{1}>0$ and $t_{2}>0$}.
\end{equation}
Now, for $s>\tau$, the law of $X_{s}$
is normal with PDF $G_{s-\tau}(z-x)$, and
\[
\big(G_{t-s}(y-z)\big)^\beta
=\beta^{-\frac{d\beta}{2}}(2\pi)^{-\frac{d(\beta-1)}{2}}t_{1}^{-\frac{d(\beta-1)}{2}}
G_{t_{1}}(z-y),
\]
where $t_{1}=\beta^{-1}(t-s)$, so that
\[
I_{\alpha,\beta}(\tau,x,s,t,y)
=\beta^{-(\frac{d}{2}+\alpha)\beta}(2\pi)^{-\frac{d(\beta-1)}{2}}t_{1}^{-(\frac{d(\beta-1)}{2}+\alpha\beta)}
\int_{\mathbb{R}^{d}}\left|z-y\right|^{\alpha\beta}G_{t_{1}}(z-y)G_{t_{2}}(z-x)\,\textrm{d}z.
\]
Then
\begin{align*}
I_{\alpha,\beta}(\tau,x,s,t,y) & =\frac{G_{t_{1}+t_{2}}(y-x)}{(2\pi)^{\frac{d(\beta-1)}{2}}\beta^{(\frac{d}{2}+\alpha)\beta}t_{1}^{\frac{d\beta-d}{2}+\alpha\beta}}\int_{\mathbb{R}^{d}}\left|z-y\right|^{\alpha\beta}
 G_{\frac{t_{1}t_{2}}{t_{1}+t_{2}}}(z-\frac{t_{1}x+t_{2}y}{t_{1}+t_{2}})\,\textrm{d}z\\
 & =\frac{G_{t_{1}+t_{2}}(y-x)}{(2\pi)^{\frac{d(\beta-1)}{2}}\beta^{\frac{d\beta}{2}+\alpha\beta}t_{1}^{\frac{d(\beta-1)}{2}+\alpha\beta}}\int_{\mathbb{R}^{d}}
 \Big|\sqrt{\frac{t_{1}t_{2}}{t_{1}+t_{2}}}z
  +\frac{t_{1}(x-y)}{t_{1}+t_{2}}\Big|^{\alpha\beta}G_{1}(z)\,\textrm{d}z,
\end{align*}
where the second step follows by changing the variable in the integral,
which leads to \eqref{eq:I-a-b-1}.
\end{proof}

\begin{lem}\label{lem:4.2}
Let $\alpha\ge 0$ and $\beta\geq1$.
\begin{enumerate}
\item[\rm (i)] For $t>s>\tau\geq0$,
\begin{align}\label{eq:I-a-b-root}
\sqrt[\beta]{I_{\alpha,\beta}(\tau,x,s,t,y)}
\leq C_{3}(t-\tau)^{\frac{d(\beta-1)}{2\beta}-\frac{\alpha}{2}}
(t-s)^{-\frac{d(\beta-1)}{2\beta}}\Big[\kappa_{1}
\big(\frac{s-\tau}{t-s}\big)^{\frac{\alpha}{2}}+\big(\frac{|x-y|}{\sqrt{t-\tau}}\big)^{\alpha}\Big]
G_{\beta(t-\tau)}(y-x),
\end{align}
where
\textup{
\begin{equation}
\kappa_{1}
=\sqrt[\beta]{\int_{\mathbb{R}^{d}}\left|z\right|^{\alpha\beta}G_{1}\left(z\right)\,\textrm{d}z},
\qquad
C_{3}=\beta^{\frac{2d+\alpha}{2}+\frac{d(\beta-1)}{2\beta}+\frac{\alpha}{2}}
 (2\pi)^{\frac{d(\beta-1)}{2\beta}}C_{1}^{\frac{1}{\beta}},\label{eq:k C-3}
\end{equation}
}
where $C_1$ given in (\ref{eq:C_1}) depends only on $\alpha\geq0$, $\beta\geq1$, and $d$.

\item[\rm (ii)] If $\alpha\geq0$ and $\beta\geq1$
such that $\frac{d(\beta-1)}{2\beta}+\frac{\alpha}{2}<1$, then
\begin{align}\label{eq:integ-01}
\int_{\tau}^{t}\sqrt[\beta]{I_{\alpha,\beta}(\tau,x,s,t,y)}\,\textrm{d}s
\leq C_{4}(t-\tau)^{1-\frac{\alpha}{2}}G_{\beta(t_{1}+t_{2})}(y-x)
\Big(1+\frac{|x-y|^\alpha}{(t-\tau)^{\frac{\alpha}{2}}}\Big)
\qquad\mbox{for $t>\tau\geq0$}, 
\end{align}
where $C_{4}=\max\left(\kappa_{3},\kappa_{0}\right)C_{3}$
depends only on $\beta,\alpha$, and $d$ $($ note that $\kappa_{3},\kappa_{0}$ and $C_3$ can be traced in the proof below$)$.
\end{enumerate}
\end{lem}

\begin{proof}
By \eqref{eq:I-a-b-1} and the triangle inequality, we obtain
\begin{align*}
\sqrt[\beta]{I_{\alpha,\beta}(\tau,x,s,t,y)}
& =C_{1}^{\frac{1}{\beta}}
t_{1}^{-(\frac{d(\beta-1)}{2\beta}+\alpha)}
\Big(\int_{\mathbb{R}^{d}}
\Big|\sqrt{\frac{t_{1}t_{2}}{t_{1}+t_{2}}}z
  +\frac{t_{1}(x-y)}{t_{1}+t_{2}}\Big|^{\alpha\beta}G_{1}(z)\,\textrm{d}z\Big)^{\frac{1}{\beta}}  \sqrt[\beta]{G_{t_{1}+t_{2}}(y-x)}.\\
 &\le C_{1}^{\frac{1}{\beta}} t_{1}^{-(\frac{d(\beta-1)}{2\beta}+\alpha)}
  \Big(\kappa_{1}\Big(\frac{t_{1}t_{2}}{t_{1}+t_{2}}\Big)^{\frac{\alpha}{2}}
    +\Big(\frac{t_{1}|x-y|}{t_{1}+t_{2}}\Big)^{\alpha}\Big)
    \sqrt[\beta]{G_{t_{1}+t_{2}}(y-x)}.
\end{align*}
Using the identity
\[
\sqrt[\beta]{G_{t_{1}+t_{2}}(y-x)}
=\beta^{\frac{d}{2}}(2\pi)^{\frac{d(\beta-1)}{2\beta}}(t_{1}+t_{2})^{\frac{d(\beta-1)}{2\beta}}G_{\beta(t_{1}+t_{2})}(y-x),
\]
and substituting this into the previous inequality, we therefore obtain
\begin{align*}
\sqrt[\beta]{I_{\alpha,\beta}(\tau,x,s,t,y)}
\leq &\, C_{2}(t_{1}+t_{2})^{\frac{d(\beta-1)}{2\beta}}(t_{1}+t_{2})^{-\frac{\alpha}{2}}
G_{\beta(t_{1}+t_{2})}(y-x)\\
 & \times\Big(\kappa_{1}
  t_{1}^{-(\frac{d(\beta-1)}{2\beta}+\frac{\alpha}{2})} t_{2}^{\frac{\alpha}{2}}
  +t_{1}^{-\frac{d(\beta-1)}{2\beta}}\frac{|x-y|^\alpha}{(t_{1}+t_{2})^{\frac{\alpha}{2}}}\Big),
\end{align*}
where
\begin{equation}
C_{2}=\beta^{\frac{d}{2}}(2\pi)^{\frac{d(\beta-1)}{2\beta}}C_{1}^{\frac{1}{\beta}}.\label{eq:C-2}
\end{equation}
Since
\begin{align*}
&t_{1}=\beta^{-1}(t-s),\quad t_{2}=s-\tau,\\
&t-\tau\geq t_{1}+t_{2}=\beta^{-1}(t-s)+s-\tau\geq\beta^{-1}(t-\tau),
\end{align*}
and
\[
G_{\beta(t_{1}+t_{2})}(y-x)\leq\beta^{\frac{d}{2}}G_{\beta(t-\tau)}(y-x),
\]
substituting these relations into the previous inequality, we thus deduce
\begin{align}
\sqrt[\beta]{I_{\alpha,\beta}(\tau,x,s,t,y)}
& \leq C_{3}(t-\tau)^{\frac{d(\beta-1)}{2\beta}-\frac{\alpha}{2}}G_{\beta(t-\tau)}(y-x)\nonumber\\
 &\quad\,  \times\Big(\kappa_{1}
 (s-\tau)^{\frac{\alpha}{2}}(t-s)^{-(\frac{d(\beta-1)}{2\beta}+\frac{\alpha}{2})}
 +(t-s)^{-\frac{d(\beta-1)}{2\beta}}\frac{|x-y|^{\alpha}}{(t-\tau)^{\frac{\alpha}{2}}}\Big),\label{3.15a}
\end{align}
where
\begin{equation}
C_{3}=C_{2}\beta^{\frac{d+\alpha}{2}+\frac{d(\beta-1)}{2\beta}+\frac{\alpha}{2}}=\beta^{\frac{2d+\alpha}{2}+\frac{d(\beta-1)}{2\beta}+\frac{\alpha}{2}}(2\pi)^{\frac{d(\beta-1)}{2\beta}}C_{1}^{\frac{1}{\beta}},\label{eq:C-3}
\end{equation}
which implies (i).

To prove (ii), we observe that the integral of the first term in the
bracket in \eqref{3.15a} can be written as
\begin{equation*}
\kappa_{1}\int_{\tau}^{t}\frac{(s-\tau)^{\frac{\alpha}{2}}}{(t-s)^{\frac{d\beta-d}{2\beta}+\frac{\alpha}{2}}}\,\textrm{d}s\\
  =\kappa_{1}\int_{\tau}^{t}\frac{(t-\tau-(t-s))^{\frac{\alpha}{2}}}{(t-s)^{\frac{d\beta-d}{2\beta}+\frac{\alpha}{2}}}\,\textrm{d}s\\
  =\kappa_{0}(t-\tau)^{1-\frac{d(\beta-1)}{2\beta}},
\end{equation*}
where
\[
\kappa_{0}=\kappa_{1}\int_{0}^{1}(1-s)^{\frac{\alpha}{2}}s^{-\frac{d\beta-d}{2\beta}-\frac{\alpha}{2}}\,\textrm{d}s,
\]
which is finite when
\[
\frac{d\beta-d}{2\beta}+\frac{\alpha}{2}<1,\quad\frac{\alpha}{2}>-1.
\]
Similarly, the second integral can be written as
\[
\int_{\tau}^{t}(t-s)^{-\frac{d(\beta-1)}{2\beta}}\,\textrm{d}s
=\kappa_{3}(t-\tau)^{1-\frac{d(\beta-1)}{2\beta}},
\]
where
$\kappa_{3}=\int_{0}^{1}s^{-\frac{d(\beta-1)}{2\beta}}\,\textrm{d}s$
is finite when $\frac{d(\beta-1)}{2\beta}<1$. Then
\begin{align*}
\int_{\tau}^{t}\sqrt[\beta]{I_{\alpha,\beta}(\tau,x,s,t,y)}\,\textrm{d}s\leq C_{3}(t-\tau)^{1-\frac{\alpha}{2}}G_{\beta(t-\tau)}(y-x)
\Big(\kappa_{0}+\kappa_{3} \frac{|x-y|^\alpha}{(t-\tau)^{\frac{\alpha}{2}}}\Big),
\end{align*}
and \eqref{eq:integ-01} follows immediately.
\end{proof}

\section{Explicit Estimates for the Fundamental Solution}

We retain the basic assumption on the vector field $b(x,t)$
that is bounded and Borel measurable.
Notice that the divergence-free assumption
on $b$ in the following theorem, Theorem \ref{prop:4.3}, is not needed.

\begin{thm}\label{prop:4.3}
For every $\beta>1$, there are constants $C_{1}$
and $C_{2}$ depending only on $\beta$ and the dimension $d$ such
that
\begin{equation}\label{exp-0-bound}
p_{b}(\tau,x,t,y)\leq C_{1}\textrm{e}^{C_{2}(t-\tau)\left\Vert b\right\Vert _{\infty}^{2}}G_{\beta(t-\tau)}(y-x)
\qquad\,\,\,\mbox{for $t>\tau\geq0$ and $x,y\in\mathbb{R}^{d}$}.
\end{equation}
\end{thm}

\begin{proof}
According to Theorem \ref{thm:3.1},
\[
p_{b}(\tau,x,t,y)= G_{t-\tau}(y-x)
 +\mathbb{E}\Big[\int_{\tau}^{t}U_{s}b(X_{s},s)\cdot\frac{y-X_{s}}{t-s}\,G_{t-s}(y-X_{s})
 \,\textrm{d}s\Big],
\]
where $X_{s}=B_{t}-B_{\tau}+x$, $B$ is the standard Brownian motion
of dimension $d$, and $U$ is the Cameron-Martin density of $L_{b}$-diffusion
with respect to the Brownian motion.
Consider the second term on the right-hand side:
\[
J\equiv\mathbb{E}\Big[\int_{\tau}^{t}U_{s}b(X_{s},s)\cdot\frac{y-X_{s}}{t-s}\,G_{t-s}(y-X_{s})
\,\textrm{d}s\Big].
\]
By the H\"older inequality, we have
\[
|J|\leq\left\Vert b\right\Vert _{\infty}\int_{\tau}^{t}
\sqrt[\gamma]{\mathbb{P}\left[U_{s}^{\gamma}\right]}\sqrt[\beta]{I_{1,\beta}(\tau,x,s,t,y)}\,\textrm{d}s,
\]
where $\frac{1}{\gamma}+\frac{1}{\beta}=1$ and
\[
U_{s}=\exp\Big[\int_{\tau\wedge s}^{s}b(X_{r},r)\cdot\textrm{d}B_{r}
  -\frac{1}{2}\int_{\tau\wedge s}^{s}|b|^{2}(X_{r},r)\,\textrm{d}r\Big].
\]
Since $U$ is an exponential martingale, then
\[
\mathbb{P}\left[U_{s}^{\gamma}\right]\leq\textrm{e}^{\frac{1}{2}\gamma(\gamma-1)\left\Vert b\right\Vert _{\infty}^{2}(s-\tau)},
\]
so that
\[
|J|\leq\left\Vert b\right\Vert _{\infty}\textrm{e}^{\frac{1}{2}(\gamma-1)\left\Vert b\right\Vert _{\infty}^{2}(t-\tau)}\int_{\tau}^{t}\sqrt[\beta]{I_{1,\beta}(\tau,x,s,t,y)}\,\textrm{d}s,
\]
where
\begin{equation}\label{eq:I-a-b-2}
I_{1,\beta}(\tau,x,s,t,y)=\mathbb{P}\Big[\Big(\frac{|y-X_{s}|}{t-s}h(s,X_{s},t,y)\Big)^{\beta}\Big]
\end{equation}
and $X_{s}=B_{s}-B_{\tau}+x$. Thus, by Lemma \ref{lem:4.2},
\[
|J|\leq C_{3}\left\Vert b\right\Vert _{\infty}\textrm{e}^{\frac{1}{2}(\gamma-1)\left\Vert b\right\Vert _{\infty}^{2}(t-\tau)}\left(\sqrt{t-\tau}+|y-x|\right)G_{\beta(t_{1}+t_{2})}(y-x),
\]
where $C_{3}>1$ depends only on $d$ and $\beta$, and
\[
t_{1}=\beta^{-1}(t-s),\quad t_{2}=s-\tau.
\]
The estimate follows from the following inequality: For every $\gamma>1$,
there are constants $C_{5}$ and $C_{6}$ depending only on $\gamma,\beta$,
and $d$ such that
\[
\sqrt{s}\left\Vert b\right\Vert_{\infty}\textrm{e}^{\frac{1}{2(\beta-1)}
\left\Vert b\right\Vert _{\infty}^{2}s}\big(1+\frac{|z|}{\sqrt{s}}\big)G_{\beta s}(z)
\leq C_{5}\textrm{e}^{C_{6}\left\Vert b\right\Vert _{\infty}^{2}s}G_{\gamma\beta s}(z).
\]
\end{proof}
As a consequence, we may deduce the following estimate that will be
used in the proof of the gradient estimate for $p_{b}(s,x,t,y)$.

\begin{lem}\label{lem:4.4}
Let $b$ be a bounded time-dependent vector
field with bound $\left\Vert b\right\Vert _{\infty}$, let $x,\xi\in\mathbb{R}^{d}$ and
$t>0$ be fixed, and let $\gamma\geq1$. Consider the following integral{\rm :}
\[
J(\varepsilon)=\sqrt[\gamma]{\int_{\mathbb{R}^{d}}\big(p_{b}(\tau,\xi,\tau+\varepsilon,z)\big)^{\gamma}
  p_{-b}(0,x,t-\varepsilon,z)\,{\rm d}z}.
\]
Then, for every $\beta>\gamma$, there are $C_{1}$ and $C_{2}$ depending
only on $\beta,\gamma$, and the dimension $d$ such that
\[
J(\varepsilon)\leq C_{1}\textrm{e}^{C_{2}t\left\Vert b\right\Vert _{\infty}^{2}}\left(\frac{t}{\varepsilon}\right)^{\frac{d(\gamma-1)}{2\gamma}}G_{\gamma\beta t}(\xi-x)
\qquad\,\,\mbox{for any $0<\varepsilon<t$ and $x,\xi\in\mathbb{R}^{d}$.}
\]
\end{lem}

\begin{proof}
Without loss of generality, we may assume that $\tau=0$.
Let $\lambda_{1}\geq1$, $\lambda_{2}\geq1$, and
\[
t_{1}=\gamma^{-1}\lambda_{1}\varepsilon,\quad t_{2}=\lambda_{2}(t-\varepsilon).
\]
Then
\[
\frac{|z-\xi|^{2}}{t_{1}}+\frac{|z-x|^{2}}{t_{2}}
=\frac{\big|z-\big(\frac{t_{2}}{t_{1}+t_{2}}\xi+\frac{t_{1}}{t_{1}+t_{2}}x\big)\big|^{2}}
{\frac{t_{1}t_{2}}{t_{1}+t_{2}}}+\frac{|\xi-x|^{2}}{(t_{1}+t_{2})},
\]
so that the product
\[
\big(G_{\lambda_{1}\varepsilon}(z-\xi)\big)^{\gamma}G_{\lambda_{2}(t-\varepsilon)}(z-x)
\]
equals
\[
\big(\frac{t_{1}+t_{2}}{t_{1}}\big)^{\frac{d}{2}(\gamma-1)}
\big(G_{\gamma(t_{1}+t_{2})}(\xi-x)\big)^{\gamma}
G_{\frac{t_{1}t_{2}}{t_{1}+t_{2}}}(z-\frac{t_{2}\xi}{t_{1}+t_{2}}-\frac{t_{1}x}{t_{1}+t_{2}}).
\]
Let $\beta>1$ and $\gamma\geq1$ be fixed. Then, by (\ref{exp-0-bound}),
\[
p_{b}(\tau,\xi,\tau+\varepsilon,z)\leq C_{1}\textrm{e}^{C_{2}\varepsilon\left\Vert b\right\Vert _{\infty}^{2}}G_{\beta\varepsilon}(z-\xi)
\]
and
\[
p_{-b}(0,x,t-\varepsilon,z)
\leq C_{1}\textrm{e}^{C_{2}(t-\varepsilon)\left\Vert b\right\Vert _{\infty}^{2}}
G_{\beta(t-\varepsilon)}(z-x),
\]
so that
\begin{align*}
E & \equiv \big(p_{b}(\tau,\xi,\tau+\varepsilon,z)\big)^{\gamma}p_{-b}(0,x,t-\varepsilon,z)\\
 & \leq C_{1}^{\gamma+1}\textrm{e}^{C_{2}(\gamma\varepsilon+t-\varepsilon)\left\Vert b\right\Vert _{\infty}^{2}}\big(G_{\beta\varepsilon}(z-\xi)\big)^{\gamma}G_{\beta(t-\varepsilon)}d(z-x)\\
 & =C_{1}^{\gamma+1}\textrm{e}^{C_{2}(\gamma\varepsilon+t-\varepsilon)
   \left\Vert b\right\Vert _{\infty}^{2}}\big(\frac{t_{1}+t_{2}}{t_{1}}\big)^{\frac{d}{2}(\gamma-1)}\\
 &\quad \times \big(G_{\gamma(t_{1}+t_{2})}(\xi-x)\big)^{\gamma}
   G_{\frac{t_{1}t_{2}}{t_{1}+t_{2}}}(z-\frac{t_{2}\xi}{t_{1}+t_{2}}-\frac{t_{1}x}{t_{1}+t_{2}}),
\end{align*}
where
\[
t_{1}=\gamma^{-1}\beta\varepsilon,\quad t_{2}=\beta(t-\varepsilon).
\]
Thus, we have
\[
J(\varepsilon)\leq C_{1}^{1+\frac{1}{\gamma}}
\textrm{e}^{\frac{C_2}{\gamma}(\gamma\varepsilon+t-\varepsilon)\left\Vert b\right\Vert_{\infty}^{2}}\big(\frac{t_{1}+t_{2}}{t_{1}}\big)^{\frac{d}{2}(1-\frac{1}{\gamma})}
G_{\gamma(t_{1}+t_{2})}(\xi-x).
\]
Since
\[
\gamma^{-1}\beta t\leq t_{1}+t_{2}=\gamma^{-1}\beta\varepsilon+\beta(t-\varepsilon)\leq\beta t,
\]
then
\[
J(\varepsilon)\leq\gamma^{\frac{d}{2}(1-\frac{1}{\gamma})}C_{1}^{1+\frac{1}{\gamma}}\textrm{e}^{C_{2}t\left\Vert b\right\Vert _{\infty}^{2}}\big(\frac{t}{\varepsilon}\big)^{\frac{d}{2}(1-\frac{1}{\gamma})}
G_{\gamma\beta t}(\xi-x),
\]
which yields the required estimate.
\end{proof}

\section{An Explicit Gradient Estimate for the Fundamental Solution}

In this section, we assume that the vector field $b(x,t)$ is smooth, divergence-free, and bounded,
and its derivative is also bounded.
Under these assumptions,
the $L_{b}$-diffusion may be constructed by solving It\^o's stochastic
differential equations. The goal of this section is to  establish an explicit
gradient estimate for $p_{b}(\tau,\xi,t+\tau,x)$
with respect to $x$.
Recall that, for the heat kernel on $\mathbb{R}^{d}$,
\begin{equation}\label{model-der-e1-1}
|\nabla_{x}G_{t}(x-\xi)|\leq\frac{\beta^{\frac{d+1}{2}}}{\sqrt{2(\beta-1)}}\frac{1}{\sqrt{t}}G_{\beta t}(x-\xi)
\qquad\,\,\mbox{for any $t>0$},
\end{equation}
where $\beta>1$ is any constant. We aim to achieve
a similar bound for $p_{b}(\tau,\xi,t+\tau,x)$.
To this end, we first prove the following
estimate.
\begin{lem}\label{lem:5-1}
Let $T>\tau\geq 0$ be fixed, and let $Y$ be the strong
solution of the Cauchy problem of the stochastic differential equations{\rm :}
\begin{equation}\label{eq:n1}
\begin{cases}
\textrm{\rm d}Y^{i}=\textrm{\rm d}B_{t}^{i}-b^{i}(Y,T+\tau-t)\textrm{\rm d}t,\\[1mm]
Y^{i}(0)=x^{i},
\end{cases}
\qquad\quad\,\mbox{for $i=1,\cdots,d$},
\end{equation}
where $B=(B^1_t, B^2_t, \dots, B^d_t)$ is the standard Brownian motion of dimension
$d$ on a probability space $(\varOmega,\mathcal{F},\mathbb{P})$.
Then $Z_{j}^{i}=\partial_{x^{j}}Y^{i}, i,j=1,\dots, n$, form the solution of the Cauchy problem of
the differential equations{\rm :}
\begin{equation}\label{eq:n2}
\begin{cases}
\textrm{\rm d}Z_{j}^{i}=-Z_{j}^{k}\frac{\partial b^{i}}{\partial y^{k}}(Y,T+\tau-t)\textrm{\rm d}t,\\[1mm]
Z_{j}^{i}(0)=\delta_{j}^{i},
\end{cases}
\end{equation}
and satisfy
\begin{equation}\label{Z-estimate}
|Z(t)|\leq|Z(0)|e^{2\left(\sqrt{T}-\sqrt{T-t}\right)\left\Vert \nabla b\right\Vert _{\tau\rightarrow\tau+T}}
\qquad\,\textrm{ for }t\in[0,T].
\end{equation}
\end{lem}

\begin{proof}
From \eqref{eq:n2}, we deduce
\begin{equation}\label{eq:n2-1}
\begin{cases}
\textrm{d}|Z|^{2}=-2Z_{i}^{j}Z_{j}^{k}\frac{\partial b^{i}}{\partial y^{k}}(Y,T+\tau-t)\,\textrm{d}t,\\[1mm]
Z_{j}^{i}(0)=\delta_{j}^{i}.
\end{cases}
\end{equation}
By definition, we have
\[
|\sqrt{s-\tau}\nabla b(x,s)|\leq\left\Vert \nabla b\right\Vert _{\tau\rightarrow\tau+T}
\qquad\textrm{ for }s\in[\tau,\tau+T],
\]
which yields that
\[
|\nabla b(x,T+\tau-s)|\leq\frac{1}{\sqrt{T-s}}\left\Vert \nabla b\right\Vert_{\tau\rightarrow\tau+T}
\qquad\textrm{ for }s\in[0,T).
\]
Therefore, we have
\begin{align*}
|Z(t)|^{2} & =d^{2}-2\int_{0}^{t}Z_{i}^{j}(s)Z_{j}^{k}(s)\frac{\partial b^{i}}{\partial y^{k}}(Y,T+\tau-s)\,\textrm{d}s\\
 & \leq d^{2}+2\left\Vert \nabla b\right\Vert _{\tau\rightarrow\tau+T}\int_{0}^{t}\frac{|Z(s)|^{2}}{\sqrt{T-s}}
 \,\textrm{d}s\qquad\,\, \mbox{for $0\leq t\leq T$}.
\end{align*}
Let
\[
f(t)=\int_{0}^{t}\frac{|Z(s)|^{2}}{\sqrt{T-s}}\,\textrm{d}s.
\]
Then the previous integral inequality may be written as
\[
|Z(t)|^{2}=f'(t)\sqrt{T-t}\leq d^{2}+2\left\Vert \nabla b\right\Vert _{\tau\rightarrow\tau+T}f(t)
\qquad \mbox{for $0\leq t\leq T$},
\]
so that
\[
f'(t)\leq\frac{d^{2}}{\sqrt{T-t}}+\frac{2\left\Vert \nabla b\right\Vert _{\tau\rightarrow\tau+T}}{\sqrt{T-t}}f(t)
\qquad \mbox{for all $0\leq t<T$}.
\]
Define
\begin{align*}
q(t):=\exp\left(-2\left\Vert \nabla b\right\Vert _{\tau\rightarrow\tau+T}\int_{0}^{t}\frac{1}{\sqrt{T-s}}\,\textrm{d}s\right)
=\exp\left(-4\left\Vert \nabla b\right\Vert _{\tau\rightarrow\tau+T}\big(\sqrt{T}-\sqrt{T-t}\big)\right).
\end{align*}
Then $q(t)$ satisfies
\[
\begin{cases}
q'(t)=-\frac{2\left\Vert \nabla b\right\Vert _{\tau\rightarrow\tau+T}}{\sqrt{T-t}}q(t),\\
q(0)=1.
\end{cases}
\]
This implies that
\[
(fq)'(t)\leq\frac{d^{2}}{\sqrt{T-t}}q(t).
\]
After integrating from $0$ to $t\leq T$, we have
\begin{align*}
f(t)
\leq d^{2}
\textrm{e}^{-4\left\Vert \nabla b\right\Vert_{\tau\rightarrow\tau+T}\sqrt{T-t}}
\int_{0}^{t}\frac{1}{\sqrt{T-s}}
\textrm{e}^{4\left\Vert \nabla b\right\Vert _{\tau\rightarrow\tau+T}\sqrt{T-s}}\,\textrm{d}s
=d^{2}\frac{\textrm{e}^{4\left\Vert \nabla b\right\Vert _{\tau\rightarrow\tau+T}(\sqrt{T}-\sqrt{T-t})}-1}{2\left\Vert b\right\Vert _{\tau\rightarrow\tau+T}}.
\end{align*}
Then
\begin{align*}
|Z(t)|^{2}
\leq d^{2}+2\left\Vert \nabla b\right\Vert _{\tau\rightarrow\tau+T}f(t)
 =d^{2}\textrm{e}^{4\left\Vert \nabla b\right\Vert _{\tau\rightarrow\tau+T}(\sqrt{T}-\sqrt{T-t})},
\end{align*}
that is,
\[
|Z(t)|\leq|Z(0)|\textrm{e}^{2\left\Vert \nabla b\right\Vert _{\tau\rightarrow\tau+T}(\sqrt{T}-\sqrt{T-t})}
\qquad\mbox{for all $t\in[0,T]$},
\]
where $|Z(0)|=d$ is the dimension, or the norm
of the identity matrix.
\end{proof}

\begin{thm}[Explicit Gradient Estimate]\label{thm:4.5}
For any $\beta>1$, there
are constants $C_{1}$ and $C_{2}$ depending only on $\beta$ and
the dimension $d$ such that
\begin{equation}
\left|\nabla p_{b}\right|(\tau,\xi,t+\tau,x)\leq\frac{C_{1}}{\sqrt{t}}\textrm{e}^{C_{2}t\left\Vert b\right\Vert _{\infty}^{2}+\frac{1}{2}\sqrt{t}\left\Vert \nabla b\right\Vert _{\tau\rightarrow\tau+t}}G_{\beta t}(\xi-x)\label{est-c1-sharp}
\end{equation}
for all $t>0$, $\tau\geq0$, and $x,\xi\in\mathbb{R}^{d}$.
\end{thm}

\begin{proof}
According to the forward Bismut-type formula (see Theorem \ref{thm:bis-f}),
\[
\partial_{x^{j}}\ln p_{b}(\tau,\xi,T+\tau,x)
=\mathbb{Q}\Big[\int_{0}^{T}\frac{\rho'(t)}{\rho(T)}Z_{j}^{k}(t)\,\textrm{d}B^{k}(t)\Big],
\]
where $\mathbb{Q}$ is the conditional law that $Y_{\tau}=\xi$ and
$Y_{T+\tau}=x$, and $\rho(t)$ can be any continuous and piecewise
differentiable function with $\rho(0)=0$ and $\rho(t)>0$ for $t>0$.
Let $\varepsilon>0$ be
small, and let $\rho$ be the function such
that $\rho(t)=t$ for $t\in[0,T-\varepsilon]$ and $\rho(t)=T-\varepsilon$
for $t>T-\varepsilon$. Then
\begin{align*}
\partial_{x^{j}}\ln p_{b}(\tau,\xi,T+\tau,x) & =\frac{1}{T-\varepsilon}\mathbb{Q}\Big[\int_{0}^{T-\varepsilon}Z_{j}^{k}(t)\,\textrm{d}B^{k}(t)\Big]
 =\frac{1}{T-\varepsilon}\mathbb{P}\Big[R_{T-\varepsilon}\int_{0}^{T-\varepsilon}Z_{j}^{k}(t)\,\textrm{d}B^{k}(t)\Big],
\end{align*}
where
\[
R_{T-\varepsilon}=\frac{p_{b}(\tau,\xi,\tau+\varepsilon,Y_{T-\varepsilon})}{p_{b}(\tau,\xi,T+\tau,x)}
\]
according to (\ref{eq:den-01}), and $Y$ is the solution of the Cauchy problem
(\ref{eq:n1}) for the stochastic differential equations. Therefore, we have
\begin{align*}
\left|\nabla\ln p_{b}\right|(\tau,\xi,T+\tau,x) & \leq\frac{C_{q}}{T-\varepsilon}\sqrt[p]{\mathbb{E}
\Big[\Big(\frac{p_{b}(\tau,\xi,\tau+\varepsilon,Y_{T-\varepsilon})}{p_{b}(\tau,\xi,T+\tau,x)}\Big)^{p}\Big]}
\sqrt[q]{\mathbb{E}\Big[\Big(\int_{0}^{T-\varepsilon}|Z(t)|^{2}\,\textrm{d}t\Big)^{\frac{q}{2}}\Big]},
\end{align*}
where $C_q$ is the constant in the Burkholder-Davis-Gundy inequality
({\it cf.} \citep{Ikeda and Watanabe 1989}) that is applied
to handle the It\^o integral with $\frac{1}{p}+\frac{1}{q}=1$.
Thanks to (\ref{Z-estimate}),
\begin{align*}
|Z(t)|  \leq|Z(0)|\textrm{e}^{2\left\Vert \nabla b\right\Vert _{\tau\rightarrow\tau+T}(\sqrt{T}-\sqrt{T-t})}
\leq|Z(0)|\textrm{e}^{2\left\Vert \nabla b\right\Vert _{\tau\rightarrow\tau+T}\frac{t}{\sqrt{T}}}
\qquad \mbox{for all $t\in[0,T]$},
\end{align*}
where $|Z(0)|=d$ is the dimension, or the norm
of the identity matrix. It follows that
\begin{align*}
\sqrt[q]{\mathbb{E}\Big[\Big(\int_{0}^{T-\varepsilon}|Z(t)|^{2}\,\textrm{d}t\Big)^{\frac{q}{2}}\Big]}
\leq|Z(0)|\sqrt[q]{\mathbb{E}\Big[\Big(\int_{0}^{T-\varepsilon}\textrm{e}^{At}
\,\textrm{d}t\Big)^{\frac{q}{2}}\Big]}
\leq|Z(0)|\sqrt{\frac{\textrm{e}^{A(T-\varepsilon)} -1}{A(T-\varepsilon)}},
\end{align*}
where $A:=\frac{2\left\Vert \nabla b\right\Vert _{\tau\rightarrow\tau+T}}{\sqrt{T}}$.

Plugging into the previous inequality, we obtain
\begin{align*}
\left|\nabla\ln p_{b}\right|(\tau,\xi,T+\tau,x)
\leq\frac{C_{q}|Z(0)|}{T-\varepsilon}\sqrt{\frac{\textrm{e}^{A(T-\varepsilon)}-1}{A(T-\varepsilon)}}
\sqrt[p]{\mathbb{E}
\Big[\Big(\frac{p_{b}(\tau,\xi,\tau+\varepsilon,Y_{T-\varepsilon})}{p_{b}(\tau,\xi,T+\tau,x)}\Big)^{p}\Big]}.
\end{align*}
The estimate is true for any $p>1$, so is for $1\leq q<\infty$. It follows
that
\begin{align*}
\left|\nabla_{x}p_{b}\right|(\tau,\xi,T+\tau,x)
\leq\frac{C_{q}|Z(0)|}{\sqrt{T-\varepsilon}}\sqrt{\frac{\textrm{e}^{A(T-\varepsilon)}-1}{A(T-\varepsilon)}}
 \sqrt[p]{\int_{\mathbb{R}^{d}}\big(p_{b}(\tau,\xi,\tau+\varepsilon,z)\big)^{p}p_{-b}(0,x,T-\varepsilon,z)\,\textrm{d}z}
\end{align*}
for any $p>1$.
By choosing $\varepsilon=\frac{T}{2}$ and using the fact that
$\frac{\textrm{e}^{x}-1}{x}\leq\textrm{e}^{x}$ for $x>0$, the previous
inequality yields that
\begin{align*}
\left|\nabla p_{b}\right|(\tau,\xi,T+\tau,x)
\leq\sqrt{2}C_{q}|Z(0)|\frac{1}{\sqrt{T}}e^{\frac{\sqrt{T}\left\Vert \nabla b\right\Vert _{\tau\rightarrow\tau+T}}{2}}
\sqrt[p]{\int_{\mathbb{R}^{d}}
 \big(p_{b}(\tau,\xi,\tau+\frac{T}{2},z)\big)^{p}p_{-b}(0,x,\frac{T}{2},z)\,\textrm{d}z}
\end{align*}
for any $p>1$. Finally, according to Lemma \ref{lem:4.4}, for any
$\beta>p$,
\begin{align*}
J(\varepsilon)  =\sqrt[p]{\int_{\mathbb{R}^{d}}\big(p_{b}(\tau,\xi,\tau+\varepsilon,z)\big)^{p}
 p_{-b}(0,x,T-\varepsilon,z)\,\textrm{d}z}
\leq K_{1}\textrm{e}^{K_{2}T\left\Vert b\right\Vert _{\infty}^{2}}\Big(\frac{T}{\varepsilon}\Big)^{\frac{d}{2}(1-\frac{1}{\gamma})}G_{\gamma\beta T}(\xi-x).
\end{align*}
Therefore, we conclude that
\[
\left|\nabla p_{b}\right|(\tau,\xi,T+\tau,x)\leq\frac{K_{3}}{\sqrt{T}}\textrm{e}^{K_{2}T\left\Vert b\right\Vert _{\infty}^{2}+\frac{1}{2}\sqrt{T}\left\Vert \nabla b\right\Vert _{\tau\rightarrow\tau+T}}G_{\gamma\beta T}(\xi-x),
\]
which yields the estimate.
\end{proof}

\section{Linearized Vorticity Equations}

In this section, we exploit
the explicit estimates in \S 3 --\S 5
to the study of
the linearized vorticity equations
for the
the Navier-Stokes equations
\eqref{t-02}
with the viscosity constant
$\nu=\frac{1}{2}$ (without loss of generality).

We make the following identification: Any tensor field $F$ defined
on $[0,1]^{3}$ satisfying the periodic condition that $F(x+\bold{k})=F(x)$,
for $i=1,2,3$, $x\in[0,1]^{3}$, and $\bold{k}\in\mathbb{Z}^{3}$,
will be identified with its periodic
extension (with period $1$) on $\mathbb{R}^{3}$.

\subsection{Linear parabolic equations arising from the vorticity equations}

Our construction of strong solutions of
the Cauchy problem \eqref{t-02}--\eqref{ID}
of the Navier-Stokes equations \eqref{t-02}
with $\nu=\frac{1}{2}$ will be based on
the vorticity equation for $\omega=\nabla\wedge u$:
\begin{equation}\label{vort-p1}
\partial_t\omega+ (u\cdot\nabla)\omega-A(u)\omega-\frac{1}{2}\Delta\omega=0,
\end{equation}
where $A(u)$ is the total derivative of $u$, which is a tensor with
its components $A(u)_{j}^{i}=\partial_{x^{j}}u^{i}$.
Although there are several formulations of the vorticity equations,
we will explain that this version serves our aims well. The crucial
observation is based on an elementary identity:
\begin{equation}
\nabla\wedge\left((u\cdot\nabla)u\right)=(u\cdot\nabla)\omega-A(u)\omega\label{eq:9-04}
\end{equation}
so that the vorticity equations are equivalent to
\begin{equation}\label{eq:9-05}
\nabla\wedge\big(\partial_t u+(u\cdot\nabla)u-\frac{1}{2}\Delta u\big)=0,
\end{equation}
which is not surprising though.
One can even argue that this is from where
the vorticity equations come. However, this formulation allows
us to define an iteration for the construction of strong solutions
of the Navier-Stokes equations.

More precisely, we start with a periodic, smooth, and divergence-free
vector field $b(x,t)$ such that $b(x,0)=u_{0}(x)$,
and we then want to construct a nonlinear mapping that sends $b(x,t)$
to $v(x,t)$, which will be denoted by $V$, so that $v=V(b)$.
This is achieved by the two steps:

\medskip
{\it Step} 1. For given $b(x,t)$,
we define a vector field $w(x,t)$ by solving the following
Cauchy problem
of the linear parabolic equations:
\begin{equation}\label{eq:9-06}
\begin{cases}
\partial_t w+(b\cdot\nabla)w-A(b)w-\frac{1}{2}\Delta w=0,\\
w(\cdot,0)=\omega_{0},
\end{cases}
\end{equation}
where $A(b)=(\partial_{x^j} b^{i})_{1\le i,j\le 3}$ is the total
derivative of $b$; this $w$ should be a candidate of the vorticity,
while it is not the vorticity of $b$.
There are two properties of the unique solution
$w(x,t)$ which are important to our approach.

First, $w$ is again
divergence-free. In fact, the divergence of $w(x,t)$, denoted by $f(x,t)$,
{\it i.e.}, $f=\nabla\cdot w$, satisfies the following parabolic equation:
\begin{equation}\label{eq:9-07}
\begin{cases}
\partial_t f+(b\cdot\nabla)f-\frac{1}{2}\Delta f=0,\\
f(\cdot,0)=0.
\end{cases}
\end{equation}
This equation may be obtained by differentiating (\ref{eq:9-06})
and using the following elementary vector identity:
\begin{equation}\label{eq:9-08}
\nabla\cdot\left((b\cdot\nabla)w\right)
=b\cdot\nabla(\nabla\cdot w)+\nabla\cdot\left(A(b)w\right),
\end{equation}
which holds for any vector
field $w$ and any $b$ such that $\nabla\cdot b=0$.
By the uniqueness of linear parabolic equation, we conclude
that $f=0$, so that $w(x,t)$ is divergence-free.

Second, we
show that the mean: $m(t)=\varint_{[0,1]^{3}}w(x,t)\,\textrm{d}x=0$.
In fact, since both $b$ and $w$ are divergence-free and periodic,
then
\[
\int_{[0,1]^{3}}(b\cdot\nabla)w\,\textrm{d}x
=\int_{[0,1]^{3}}A(b)w\,\textrm{d}x=\frac{1}{2}\int_{[0,1]^{3}}\Delta w\,\textrm{d}x=0,
\]
so that $m(t)=m(0)=0$.

\medskip
{\it Step} 2. We define the candidate for the velocity $v(x,t)$ by solving
the Poisson equation:
\begin{equation}\label{6.10a}
\Delta v=-\nabla\wedge w
\end{equation}
at any instance such
that $\int_{[0,1]^{3}}v(x,\cdot)\,\textrm{d}x=0$, which has a unique
solution satisfying the periodic condition.
Applying $\nabla\cdot$
to both sides of the Poisson equation, we obtain that $\Delta\left(\nabla\cdot v\right)=0$,
so that $\nabla\cdot v$ must be constant for every $t$.
Since $\int_{[0,1]^{3}}\nabla\cdot v\, {\rm d}x=0$,
then $\nabla\cdot\nu=0$. Therefore, $\nabla\cdot v=0$ and $\nabla\cdot w=0$.
The Poisson equation for $v$ implies that
\begin{align*}
&\nabla\wedge\left(\nabla\wedge v-w\right)=-\Delta v-\nabla\wedge w=0,\\
&\nabla\cdot\left(\nabla\wedge v-w\right)=-\nabla\cdot w=0.
\end{align*}
Then, according to the Hodge theory, $\nabla\wedge v-w$ vanishes identically,
so it follows that $w=\nabla\wedge v$. In particular, $v(x,0)=u_{0}(x)$.

Therefore, in this way, we are able to construct an iteration via the nonlinear mapping $V$
so that $v=V(b)$.

\smallskip
The advantage for using this iteration via the nonlinear mapping
$b\rightarrow v=V(b)$ can now be put forward in the following: Observe
that the parabolic equation (\ref{eq:9-06}) can be rewritten as
\[
\begin{cases}
\left(\partial_t-L_{-b}\right)w=A(b)w,\\
w(\cdot,0)=\omega_{0},
\end{cases}
\]
where, as before, $L_{-b}$ denotes the time-dependent elliptic operator
$\frac{1}{2}\Delta-b\cdot\nabla$. The crucial difference from the
linearised Navier-Stokes equations lies in the fact that the term on the
right-hand side, $A(b)w$, is a linear zero-order term. This crucial
difference allows us to apply the Feynman-Kac-type formula to obtain the
necessary \emph{a priori} estimates, which will be derived in the
remainder of this section.

\subsection{Brownian motion on $\mathbb{T}^{3}$}

Let $h(\tau,x,t,y)$ denote the heat kernel on $\mathbb{T}^{3}$,
which is the transition probability density function of the Brownian
motion on $\mathbb{T}^{3}$, a diffusion on $\mathbb{T}^{3}$ with
its infinitesimal generator $\frac{1}{2}\Delta$. Then
\begin{equation}\label{hk-01}
h(\tau,x,t,y)=\sum_{\boldsymbol{k}\in\mathbb{Z}^{3}}G_{t-\tau}(y-x+\boldsymbol{k})
\qquad \mbox{for $x,y\in[0,1]^{3}$},
\end{equation}
where the series on the right-hand side and
its derivative series indeed converge uniformly for $(x,y)\in[0,1]^{3}$.

Since $\mathbb{T}^{3}$ is a compact manifold without boundary, there
is a Green function $C(x,y)$ associated with the Laplacian $\Delta$,
denoted by $C(x,y)$ ({\it cf.}
\citep[page 108]{Aubin1998}), which possesses
the following properties:

\begin{enumerate}
\item[\rm (i)] $C(x,y)$ is smooth out off the diagonal, periodic in $x,y\in[0,1]^{3}$, and
\[
|\nabla_{x}^{k}C(x,y)|\leq\frac{C_{1}}{|x-y|^{k}} \qquad\mbox{for $x,y\in[0,1]^{3}$},
\]
where $k=0,1,2$, $C_{1}$ is a constant, and $\int_{[0,1]^{3}}C(x,y)\,\textrm{d}y=0$.
In particular,
\[
\int_{[0,1]^{3}}|\nabla_{x}^{k}C(x,y)|\,\textrm{d}y\leq C_{2}
\qquad\,\,\mbox{for all $x\in\mathbb{R}^{3}$ and $k=0,1,2$},
\]
where $C_{2}$ is a universal constant.

\item[\rm (ii)]
For every periodic $C^{2}$ function $\psi$, the Green formula holds:
\[
\psi(x)=\int_{[0,1]^{3}}\psi(y)\,\textrm{d}y
 +\int_{[0,1]^{3}}C(x,y)\Delta\psi(y)\,\textrm{d}y.
\]
\end{enumerate}

It follows that, for any periodic function $f$ with mean zero: $\int_{[0,1]^{3}}f(y)\,\textrm{d}y=0$,
the Poisson equation:
\[
\Delta\psi=f
\]
has a unique periodic solution with mean zero, given by the Green formula:
\[
\psi(x)=\int_{[0,1]^{3}}C(x,y)f(y)\,\textrm{d}y \qquad\,\, \mbox{for $x\in\mathbb{R}^{3}$}.
\]

Let $b(\cdot,t)$ be a vector field on $[0, 1]^3$
depending on $t\geq0$, which is identified with a vector field
\[
b(x,t)=(b^{1}(x,t),b^{2}(x,t),b^{3}(x,t)) \qquad\mbox{ on $\mathbb{R}^{3}$}
\]
with period $1$ along each coordinate.
Let $\mathcal{L}_{b}=\frac{1}{2}\Delta+b(\cdot,t)\cdot\nabla$
be an elliptic operator on $[0, 1]^3$.
Let $h_{b}(\tau,x,t,y)$
denote the transition probability density function of the $\mathcal{L}_{b}$-diffusion
on $[0, 1]^3$.
Then
\[
h_{b}(\tau,x,t,y)=\sum_{\boldsymbol{k}\in\mathbb{Z}^{3}}p_{b}(\tau,x,t,y+\boldsymbol{k})
\qquad\mbox{for $x,y\in[0,1]^{3}$},
\]
where $p_{b}(\tau,x,t,y)$ is the transition
probability density function on $\mathbb{R}^{3}$ with infinitesimal
generator $L_{b}=\frac{1}{2}\Delta+b(\cdot,t)\cdot\nabla$ on $\mathbb{R}^{3}$.

\subsection{\emph{A priori} estimates for the linearized vorticity equations}

In this subsection, we establish several {\it a priori} estimates for the first step of the
iteration for solving the vorticity equation, defined in the
following: Let $u_{0}$ be a given smooth, periodic initial velocity vector
field with divergence-free, and $\omega_{0}=\nabla\wedge u_{0}$,
and let $b(x,t)$ be a smooth
time-dependent divergence-free periodic vector field on $\mathbb{R}^{3}$
with $b(x,0)=u_{0}(x)$.
Hence, $u_{0}(x+\bold{k})=u_{0}(x)$ and
$b(x+\bold{k},t)=b(x,t)$ for $\bold{k}\in\mathbb{Z}^{3}$
and $x\in\mathbb{R}^{3}$.
Define the periodic vorticity field $w(x,t)$
to be the unique solution of the Cauchy problem \eqref{eq:9-06}
for the linear vorticity equations.

The vector field $v(x,t)$ is the unique solution with mean zero to
the Poisson equation \eqref{6.10a} on $\mathbb{T}^{3}$,
which may be given by the Green formula:
\[
v^{i}(x,t)=\int_{[0,1]^{3}}\varepsilon^{ijk}C(x,y)\partial_{y^{j}}w^{k}(y,t)\,\textrm{d}y
\qquad \mbox{for $t\geq0$ and $x\in\mathbb{T}^{3}$}.
\]
We have shown that both $v$ and
$w$ are divergence-free, and $w=\nabla\times v$.

Moreover,
the solution $w(x,t)$ of the Cauchy problem \eqref{eq:9-06}
satisfies
the following implicit equation:
\begin{equation}
w(x,t)=\int_{[0,1]^{3}}h_{-b}(\tau,\xi,t,x)w(\xi,\tau)\,\textrm{d}\xi
+\int_{\tau}^{t}\int_{[0,1]^{3}}h_{-b}(s,\xi,t,x)A(b)(\xi,s)w(\xi,s)
\,\textrm{d}\xi\textrm{d}s\label{linear-re1}
\end{equation}
for $t>\tau\geq0$ and $x\in[0,1]^{3}$. On the other hand, by the
forward Feynman-Kac's formula \eqref{eq:F-K-f01}, we have the following
nonlinear representation:
\begin{equation}\label{a-22-4}
w^{k}(x,t)=\int_{[0,1]^{3}}\omega_{0}^{j}(\xi)h_{b}(0,\xi,t,x)\mathbb{P}^{\xi}
\big[Q_{j}^{k}(0,t)|X_{t}=x\big]\,\textrm{d}\xi  \qquad\mbox{for $k=1,2,3$},
\end{equation}
where $X$ is the canonical process on the path space
$\varOmega=C([0,\infty),\mathbb{T}^{3})$, $\mathbb{P}^{\xi}=\mathbb{P}^{0,\xi}$ is
the diffusion family on $\varOmega$ with infinitesimal generator
${\mathcal{L}}_{b}$, and $Q(s)=Q(s,t)$ is the solution to
\begin{equation}\label{a-22-6}
\begin{cases}
\frac{\textrm{d}}{\textrm{d}s}Q_{j}^{i}(s)=-Q_{k}^{i}(s)A(b)_{j}^{k}(X_{s},s)\qquad\,\,\mbox{ for $s\leq t$},\\[1mm]
Q_{j}^{i}(t)=\delta_{j}^{i}.
\end{cases}
\end{equation}

We now
derive several \emph{a priori} estimates.

\begin{lem}\label{lem6.1}
Let $\eta=(\eta(s))_{s\geq0}$ be any continuous curve in $\mathbb{R}^{3}$,
and let $G(s)=(G_{j}^{i}(s))_{s\leq t}$ be the unique
solution of the Cauchy problem for the ordinary differential equations{\rm :}
\begin{equation}\label{eq:X-e2-1-1}
\begin{cases}
\frac{\textrm{\rm d}}{\textrm{\rm d}s}G_{j}^{i}(s)=-G_{k}^{i}(s)A(b)_{j}^{k}(\eta(s),s)
   \qquad\,\mbox{ for $0\leq s\leq t$},\\[1mm]
G_{j}^{i}(t)=\delta_{j}^{i},
\end{cases}
\end{equation}
for $i,j=1,2,3$. Then
\begin{equation}
\sum_{i,j}G_{j}^{i}(0)G_{j}^{i}(0)
\leq9\textrm{e}^{4\sqrt{t}\left\Vert \nabla b\right\Vert _{0\rightarrow t}}.\label{a-22-03}
\end{equation}
\end{lem}

\begin{proof}
For simplicity of notation, we set
$$
f(s):=\sum_{i,j}G_{j}^{i}(s)G_{j}^{i}(s),
$$
which is the squared Hilbert-Schmidt's norm of $(G_{j}^{i}(s))$.
Then
\begin{align*}
\frac{\textrm{d}f(s)}{\textrm{d}s} & =2G_{j}^{i}(s)\frac{\textrm{d}}{\textrm{d}s}G_{j}^{i}(s)\\
 & =-2G_{k}^{i}(s)G_{j}^{i}(s)A(b)_{j}^{k}(\eta(s),s)\\
 & \geq-2|\nabla b(\eta(s),s)|f(s)\\
 & \geq-\frac{2}{\sqrt{s}}\left\Vert \nabla b\right\Vert _{0\rightarrow t}f(s)
 \qquad \mbox{ for $s\in[0,t]$}.
\end{align*}
After integration, we obtain
\begin{align*}
\ln f(t)-\ln f(0)
\geq-2\left\Vert \nabla b\right\Vert _{0\rightarrow t}\int_{0}^{t}\frac{1}{\sqrt{s}}\,\textrm{d}s
 =-4\sqrt{t}\left\Vert \nabla b\right\Vert _{0\rightarrow t}.
\end{align*}
Then
\[
G(0)\leq9\textrm{e}^{4\sqrt{t}\left\Vert \nabla b\right\Vert _{0\rightarrow t}}
\qquad\mbox{for all $t\geq0$}.
\]
\end{proof}

\begin{lem}\label{lem:6.2}
For every $\beta>1$, there are universal constants
$C_{1}$ and $C_{2}$ depending only on $\beta$ such that
\begin{equation}\label{a-22-8}
|w(x,t)|\leq C_{1}\textrm{e}^{C_{2}t\left\Vert b\right\Vert _{\infty}^{2}+2\sqrt{t}\left\Vert \nabla b\right\Vert _{0\rightarrow t}}G_{\beta t}(|\omega_{0}|)(x)
\qquad\,\mbox{for all $t>0$ and $x\in \mathbb{T}^{3}$}.
\end{equation}
\end{lem}

\begin{proof}
The estimate in Lemma \ref{lem6.1} allows us to control the conditional
expectation in \eqref{a-22-4} and to obtain
\begin{align*}
|w(x,t)| & \leq3\textrm{e}^{2\sqrt{t}\left\Vert \nabla b\right\Vert _{0\rightarrow t}}\int_{[0,1]^{3}}|\omega_{0}(\xi)|h_{b}(0,\xi,t,x)\,\textrm{d}\xi\\
 & =3\textrm{e}^{2\sqrt{t}\left\Vert \nabla b\right\Vert _{0\rightarrow t}}\int_{[0,1]^{3}}|\omega_{0}(\xi)|
  \sum_{\bold{k}\in\mathbb{Z}^{3}}p_{b}(0,\xi,t,x+\bold{k})\,\textrm{d}\xi.
\end{align*}
Using the uniform estimate ({\it cf.} Theorem \ref{prop:4.3}):
\[
p_{b}(0,\xi,t,x)\leq C_{1}\textrm{e}^{C_{2}t\left\Vert b\right\Vert _{\infty}^{2}}G_{\beta t}(x-\xi),
\]
we obtain
\begin{align*}
|w(x,t)| & \leq3\textrm{e}^{2\sqrt{t}\left\Vert \nabla b\right\Vert _{0\rightarrow t}}\int_{[0,1]^{3}}|\omega_{0}(\xi)|
 \sum_{\bold{k}\in\mathbb{Z}^{3}}G_{\beta t}(x+\bold{k}-\xi)\,\textrm{d}\xi\\
 & =3C_{1}\textrm{e}^{C_{2}t\left\Vert b\right\Vert _{\infty}^{2}+2\sqrt{t}\left\Vert \nabla b\right\Vert _{0\rightarrow t}}\int_{[0,1]^{3}}|\omega_{0}(\xi)|\sum_{\bold{k}\in\mathbb{Z}^{3}}G_{\beta t}(x+\bold{k}-\xi)\,\textrm{d}\xi\\
 & =3C_{1}\textrm{e}^{C_{2}t\left\Vert b\right\Vert _{\infty}^{2}
   +2\sqrt{t}\left\Vert \nabla b\right\Vert _{0\rightarrow t}}
    \int_{\mathbb{R}^{3}}|\omega_{0}(\xi)|G_{\beta t}(x-\xi)\,\textrm{d}\xi\\
 & =3C_{1}\textrm{e}^{C_{2}t\left\Vert b\right\Vert _{\infty}^{2}
   +2\sqrt{t}\left\Vert \nabla b\right\Vert _{0\rightarrow t}}G_{\beta t}(|\omega_{0}|)(x).
\end{align*}
\end{proof}

With Lemmas 6.1--6.2, we now establish the following three theorems.
\begin{thm}\label{thm:6.3}
For every $\beta>1$, there are two positive constants
$C_{1}$ and $C_{2}$ depending only on $\beta$ such that
\[
|\nabla w(x,t)|\leq\frac{C_{1}}{\sqrt{t}}\textrm{e}^{C_{2}t\left\Vert b\right\Vert _{\infty}^{2}+3\sqrt{t}\left\Vert \nabla b\right\Vert _{0\rightarrow t}}G_{\beta t}(|\omega_{0}|)(x)
\qquad\,\mbox{for all $t>0$ and $x\in \mathbb{T}^{3}$.}
\]
\end{thm}

\begin{proof}
Recall that $w$ satisfies the following equality:
\[
w(x,t)=\int_{[0,1]^{3}}h_{-b}(0,\xi,t,x)\omega_{0}(\xi)\,\textrm{d}\xi
     +\int_{0}^{t}\int_{[0,1]^{3}}h_{-b}(s,\xi,t,x)A(b)(\xi,s)w(\xi,s)\,\textrm{d}\xi\textrm{d}s
\]
for $t>0$ and $x\in[0,1]^{3}$. Differentiating both sides in $x$
to obtain
\begin{align*}
\nabla w(x,t) & =\int_{[0,1]^{3}}\nabla_{x}h_{-b}(0,\xi,t,x)\omega_{0}(\xi)\,\textrm{d}\xi\\
 &\quad +\int_{0}^{t}\int_{[0,1]^{3}}\nabla_{x}h_{-b}(s,\xi,t,x)A(b)(\xi,s)w(\xi,s)
   \,\textrm{d}\xi\textrm{d}s\\
 & =\int_{[0,1]^{3}}\sum_{\bold{k}\in\mathbb{Z}^{3}}\nabla_{x}p_{-b}(0,\xi,t,x+\bold{k})
   \omega_{0}(\xi)\,\textrm{d}\xi\\
 &\quad +\int_{0}^{t}\int_{[0,1]^{3}}\sum_{\bold{k}\in\mathbb{Z}^{3}}
 \nabla_{x}p_{-b}(s,\xi,t,x+\bold{k})A(b)(\xi,s)w(\xi,s)\,\textrm{d}\xi\textrm{d}s.
\end{align*}
According to Theorem \ref{thm:4.5}, for every $\beta>1$, there
are two universal constants $C_{2}$ and $C_{3}$ such that
\begin{equation}\label{o}
\left|\nabla_{x}p_{b}\right|(s,\xi,t,x)\leq C_{3}\frac{\textrm{e}^{C_{2}(t-s)\left\Vert b\right\Vert _{\infty}^{2}+\frac{1}{2}\sqrt{t-s}\left\Vert \nabla b\right\Vert _{s\rightarrow t}}}{\sqrt{t-s}}G_{\beta(t-s)}(x-\xi)\qquad\,\,\mbox{for any $t>s\geq0$}.
\end{equation}
Thanks to this estimate, using the triangle inequality, we obtain
\begin{align*}
|\nabla w|(x,t)
& \leq\int_{[0,1]^{3}}|\omega_{0}(\xi)|\sum_{\bold{k}\in\mathbb{Z}^{3}}|
   \nabla_{x}p_{-b}|(0,\xi,t,x+\bold{k})\,\textrm{d}\xi\\
&\quad +\int_{0}^{t}\int_{[0,1]^{3}}|A(b)(\xi,s)||w(\xi,s)|
   \sum_{\bold{k}\in\mathbb{Z}^{3}}|\nabla_{x}p_{-b}|(s,\xi,t,x+\bold{k})\,\textrm{d}\xi\textrm{d}s\\
 & \leq C_{3}\frac{\textrm{e}^{\frac{\sqrt{t}}{2}
  \left\Vert \nabla b\right\Vert_{0\rightarrow t}+C_{2}t\left\Vert
   b\right\Vert _{\infty}^{2}}}{\sqrt{t}}\int_{[0,1]^{3}}|\omega_{0}(\xi)|
    \sum_{\bold{k}\in\mathbb{Z}^{3}}G_{\beta t}(\xi-x-\bold{k})\,\textrm{d}\xi\\
 &\quad +C_{3}\int_{0}^{t}\frac{\textrm{e}^{\frac{\sqrt{t-s}}{2}
  \left\Vert \nabla b\right\Vert _{s\rightarrow t}
   +C_{2}(t-s)\left\Vert b\right\Vert _{\infty}^{2}}}{\sqrt{t-s}}\int_{[0,1]^{3}}|A(b)(\xi,s)||w(\xi,s)|
    \sum_{\bold{k}\in\mathbb{Z}^{3}}G_{\beta(t-s)}(\xi-x-\bold{k})\,\textrm{d}\xi\textrm{d}s\\
 & =C_{3}\frac{\textrm{e}^{\frac{\sqrt{t}}{2}\left\Vert \nabla b\right\Vert _{0\rightarrow t}
   +C_{2}t\left\Vert b\right\Vert _{\infty}^{2}}}{\sqrt{t}}\int_{\mathbb{R}^{3}}
    |\omega_{0}(\xi)|G_{\beta t}(\xi-x)\,\textrm{d}\xi\\
 &\quad +C_{3}\int_{0}^{t}\frac{\textrm{e}^{\frac{1}{2}\sqrt{t-s}
  \left\Vert \nabla b\right\Vert _{s\rightarrow t}
   +C_{2}\left\Vert b\right\Vert _{\infty}^{2}(t-s)}}{\sqrt{t-s}}\int_{\mathbb{R}^{3}}|A(b)(\xi,s)||w(\xi,s)|
    G_{\beta(t-s)}(\xi-x)\,\textrm{d}\xi\textrm{d}s.
\end{align*}
By definition of $S=S(b)$, we have
\[
\left|A(b)(\xi,s)w(\xi,s)\right|
\leq\frac{1}{\sqrt{s}}\left\Vert \nabla b\right\Vert _{0\rightarrow t}|w(\xi,s)|.
\]
Together with estimate (\ref{a-22-8}), we then deduce
\[
\left|A(b)(\xi,s)w(\xi,s)\right|
\leq \frac{3C_3}{\sqrt{s}}\left\Vert \nabla b\right\Vert _{0\rightarrow t}\textrm{e}^{C_{2}s
  \left\Vert b\right\Vert _{\infty}^{2}
  +2\sqrt{s}\left\Vert \nabla b\right\Vert _{0\rightarrow s}}G_{\beta s}(|\omega_{0}|)(\xi)
\qquad\,\mbox{for all $t>s>0$}.
\]
It follows that
\begin{align*}
I_{2}(x,t)
& =\int_{0}^{t}\frac{\textrm{e}^{\frac{1}{2}\sqrt{t-s}
  \left\Vert \nabla b\right\Vert _{s\rightarrow t}
   +C_{2}\left\Vert b\right\Vert _{\infty}^{2}(t-s)}}{\sqrt{t-s}}
   \Big(\int_{\mathbb{R}^{3}}|A(b)(\xi,s)||w(\xi,s)|
     G_{\beta(t-s)}(\xi-x)\,\textrm{d}\xi\Big)\textrm{d}s\\
 & \leq3C_{3}\left\Vert \nabla b\right\Vert _{0\rightarrow t}
  G_{\beta t}(|\omega_{0}|)(x)
  \int_{0}^{t}\frac{\textrm{e}^{\frac{1}{2}\sqrt{t-s}
   \left\Vert \nabla b\right\Vert _{s\rightarrow t}
   +2\sqrt{s}\left\Vert \nabla b\right\Vert _{0\rightarrow s}
   +C_{2}t\left\Vert b\right\Vert _{\infty}^{2}}}{\sqrt{t-s}\sqrt{s}}\,\textrm{d}s\\
 & \leq3C_{3}\left\Vert \nabla b\right\Vert _{0\rightarrow t}
   G_{\beta t}(|\omega_{0}|)(x)\textrm{e}^{C_{2}
   \left\Vert b\right\Vert _{\infty}^{2}t
   +3\sqrt{t}\left\Vert \nabla b\right\Vert _{0\rightarrow t}}
   \int_{0}^{t}\frac{1}{\sqrt{t-s}}\frac{1}{\sqrt{s}}\,\textrm{d}s\\
 & \leq C_{4}G_{\beta t}(|\omega_{0}|)(x)
  \left\Vert \nabla b\right\Vert _{0\rightarrow t}\textrm{e}^{C_{2}\left\Vert b\right\Vert _{\infty}^{2}t
  +\frac{5}{2}\sqrt{t}\left\Vert \nabla b\right\Vert _{0\rightarrow t}}\\
 & \leq \frac{C_5}{\sqrt{t}}\textrm{e}^{C_{2}t
  \left\Vert b\right\Vert _{\infty}^{2}
  +3\sqrt{t}\left\Vert \nabla b\right\Vert _{0\rightarrow t}}G_{\beta t}(|\omega_{0}|)(x).
\end{align*}
Similarly, we have
\begin{align*}
I_{1}(x,t)
 =\int_{\mathbb{R}^{d}}\left|\nabla h_{-b}\right|(0,\xi,t,x)|\omega_{0}(\xi)|\,\textrm{d}\xi \leq\frac{C_{1}}{\sqrt{t}}
 \textrm{e}^{\sqrt{t}\left\Vert \nabla b\right\Vert _{0\rightarrow t}+C_{2}t\left\Vert b\right\Vert_{\infty}^{2}}
 G_{\beta t}(|\omega_{0}|)(x).
\end{align*}
The claimed estimate follows immediately.
\end{proof}

\begin{thm}
Let $v$ be the unique solution with mean zero of the Poisson equation{\rm :}
\[
\Delta v=-\nabla\wedge w  \qquad\,\,\mbox{in $\mathbb{R}^{3}$},
\]
which is also periodic.
Then there are constants
$C_{1}$ and $C_{2}$, depending only on $\beta$, such that
\begin{align*}
&|v(x,t)|\leq C_{1}\textrm{e}^{C_{2}t
  \left\Vert b\right\Vert _{\infty}^{2}
  +2\sqrt{t}\left\Vert \nabla b\right\Vert _{0\rightarrow t}}\left\Vert\omega_{0}\right\Vert _{\infty},\\
&|\nabla v(x,t)|\leq
\frac{C_{1}}{\sqrt{t}}
\textrm{e}^{C_{2}t
  \left\Vert b\right\Vert _{\infty}^{2}
  +3\sqrt{t}\left\Vert \nabla b\right\Vert _{0\rightarrow t}}\left\Vert \omega_{0}\right\Vert _{\infty},
\end{align*}
for all $t\geq0$ and $x\in[0,1]^{3}$.
\end{thm}

\begin{proof}
By Green's formula,
\begin{align*}
&v(x,t)=-\int_{[0,1]^{3}}C(x,y)\nabla\wedge w(y,t)\,{\rm d}t,\\
&\nabla_{x}v(x,t)=\int_{[0,1]^{3}}\nabla_{x}C(x,y)\nabla\wedge w(y,t)\,{\rm d}t,
\end{align*}
where $C(x,y)$ is the Green function of $\mathbb{T}^{3}$.
Hence,
we conclude that there exist universal constants $C_{4}$ and $C_{5}$ such that
\begin{align*}
&|v(x,t)|\leq C_{4}\textrm{e}^{2\sqrt{t}\left\Vert \nabla b\right\Vert _{0\rightarrow t}+C_{2}t\left\Vert b\right\Vert _{\infty}^{2}}\left\Vert \omega_{0}\right\Vert _{\infty},\\
&|\nabla v(x,t)|\leq
\frac{C_{5}}{\sqrt{t}}
C_5\textrm{e}^{3\sqrt{t}\left\Vert \nabla b\right\Vert _{0\rightarrow t}
 +C_{2}t\left\Vert b\right\Vert_{\infty}^{2}}\left\Vert \omega_{0}\right\Vert _{\infty},
\end{align*}
as $G_{\beta t}(|\omega_{0}|)(x)\leq\left\Vert \omega_{0}\right\Vert _{\infty}$.
\end{proof}

\begin{thm}
There are two universal constants $C_{1},C_{2}>0$ such that,
if
\begin{align*}
\left\Vert b\right\Vert _{L^{\infty}([0,T]\times\mathbb{R}^{3})}\leq C_{2}\left\Vert \omega_{0}\right\Vert _{\infty},
\qquad
\left\Vert \nabla b\right\Vert _{0\rightarrow T}\leq C_{2}\left\Vert \omega_{0}\right\Vert _{\infty},
\end{align*}
then, for
\[
T=\frac{C_{1}}{\left\Vert \omega_{0}\right\Vert _{\infty}^{2}},
\]
the following estimates hold{\rm :}
\begin{align*}
&\left\Vert v\right\Vert _{L^{\infty}([0,T]\times\mathbb{R}^{3})}\leq C_{2}\left\Vert \omega_{0}\right\Vert _{\infty},\qquad\,\,\,
\left\Vert w\right\Vert _{L^{\infty}([0,T]\times\mathbb{R}^{3})}
\leq C_{2}\left\Vert \omega_{0}\right\Vert _{\infty},\\
&\left\Vert \nabla v\right\Vert_{0\rightarrow T}
  \le C_{2}\left\Vert \omega_{0}\right\Vert _{\infty},
  \qquad\,\,\, \qquad
\left\Vert \nabla w\right\Vert _{0\rightarrow T}\leq C_{2}\left\Vert \omega_{0}\right\Vert _{\infty}.
\end{align*}
\end{thm}

\begin{proof}
Let us choose $T>0$ such that
\[
C_{0}\lambda e^{C_{2}T\mu^{2}C_{0}^{2}\lambda^{2}+3\sqrt{T}\mu C_{0}\lambda}=\mu C_{0}\lambda.
\]
Then we solve $T$ to obtain
\[
T\mu^{2}C_{0}^{2}\lambda^{2}+\frac{3}{C_{2}}\sqrt{T}\mu C_{0}\lambda-\frac{1}{C_{2}}\ln\mu=0,
\]
so that
\[ 
\sqrt{T}=\frac{2\ln\mu}{C_{0}C_{2}\mu\left(\sqrt{9+4\ln\mu}+3\right)\lambda},
\]
where
$C_{0}=C_{4}\vee C_{5}\vee\ldots$ and
$\lambda=\max G_{\beta T}(|\omega_{0}|)(x)$.
Choose $\mu=e$. Then
\[
\sqrt{T}=\frac{2}{C_{0}C_{2}e\left(\sqrt{13}+3\right)\lambda}
\]
so that $T=\frac{C}{\lambda^{2}}$.
\end{proof}

\section{Navier-Stokes Equations}

Thanks to the explicit {\it a priori} estimates established in \S 3--\S6,
we are now in a position to study the strong solutions of the Cauchy problem:
\begin{equation}\label{ns-p1}
\begin{cases}
\partial_t u+ (u\cdot\nabla) u-\frac{1}{2}\Delta u=-\nabla P, \\
\nabla\cdot u=0,
\end{cases}
\end{equation}
with initial data:
\begin{equation}\label{ns-p-ID}
\quad u(x,0)=u_{0}(x),
\end{equation}
subject to the periodic condition that $u_0(x+e_{i})=u_0(x)$ for
$x\in[0,1]^{3}$, where $i=1,2,3$, $(e_{i})$ is the standard basis
of $\mathbb{R}^{3}$.

\begin{thm}
There are universal constants $C_{1}>0$ and $C_{2}>0$ such that, for
any periodic initial data $u_{0}$ with mean zero and $\omega_{0}=\nabla\wedge u_{0}$,
there exists a unique strong solution $u(x,t)$ of the Cauchy problem \eqref{ns-p1}--\eqref{ns-p-ID}
for the Navier-Stokes equations \eqref{ns-p1}
with periodic initial data \eqref{ns-p-ID} for $t\leq T$ so that
$T=\frac{C_{1}}{\left\Vert \omega_{0}\right\Vert _{\infty}^{2}}$,
and
\begin{align*}
&\left\Vert u\right\Vert _{L^{\infty}([0,T]\times\mathbb{R}^{3})}\leq C_{2}\left\Vert \omega_{0}\right\Vert _{\infty},\qquad
\left\Vert \omega\right\Vert_{L^{\infty}([0,T]\times\mathbb{R}^{3})}
\leq C_{2}\left\Vert \omega_{0}\right\Vert _{\infty},\\
&\left\Vert \nabla u\right\Vert _{0\rightarrow T}\leq C_{2}\left\Vert \omega_{0}\right\Vert _{\infty},
\qquad \qquad
\left\Vert \nabla\omega\right\Vert_{0\rightarrow T}
\leq C_{2}\left\Vert \omega_{0}\right\Vert _{\infty},
\end{align*}
where
$\left\Vert V\right\Vert _{0\rightarrow T}=\sup_{(x,t)\in\mathbb{R}^{3}\times[0,T]}\left|\sqrt{t}V(x,t)\right|$.
\end{thm}

\begin{proof}
To construct the strong solution of the Cauchy problem
\eqref{ns-p1}--\eqref{ns-p-ID},
we construct the following iterations:

\medskip
Set
$u^{(0)}(x,t)=u_{0}(x)$
for all $x$ and $t\geq0$, and define $u^{(n)}=V(u^{(n-1)})$ inductively for $n\ge 1$.
Then $\nabla\cdot u^{(n)}=0$ and $\nabla\cdot w^{(n)}=0$, and
\[
\omega^{(n)}=\nabla\wedge u^{(n)}=w^{(n)}
\qquad\,\mbox{for all $n=1,2,\cdots$}.
\]
Hence, for each $n\geq 1$, $\omega^{(n)}$ solves the linear
parabolic equations on the
torus
\[
\begin{cases}
\partial_t \omega^{(n)}+(u^{(n-1)}\cdot\nabla)\omega^{(n)}-A(u^{(n-1)})\omega^{(n)}
-\frac{1}{2}\Delta\omega^{(n)}=0,\\
\omega^{(n)}(\cdot,0)=\omega_{0},
\end{cases}
\]
where $\omega^{(n)}=\nabla\wedge u^{(n)}$
and $\nabla\cdot u^{(n)}=0$.

\medskip
Let $T=\frac{C_{1}}{\left\Vert \omega_{0}\right\Vert _{\infty}^{2}}$ in Theorem 6.5.

Then
\begin{align*}
&|u^{(n)}(x,t)|\leq C_{2}\left\Vert \omega_{0}\right\Vert _{\infty},\qquad
|\nabla u^{(n)}(x,t)|
\leq \frac{C_{2}}{\sqrt{t}}
\left\Vert \omega_{0}\right\Vert _{\infty},\\
&|\omega^{(n)}(x,t)|\leq C_{2}\left\Vert \omega_{0}\right\Vert _{\infty},
 \qquad |\nabla\omega^{(n)}(x,t)|\leq\frac{C_{2}}{\sqrt{t}}\left\Vert \omega_{0}\right\Vert _{\infty},
\end{align*}
where $C_{1}$ and $C_{2}$ are universal positive constants. By the
standard parabolic regularity theory ({\it cf.} \citep{Ladyzenskaja Solonnikov and Uralceva 1968}),
\[
\sup_{\mathbb{R}^{3}\times[\delta,T]}\big|\partial_{t}^{k}\nabla^{l}u^{(n)}\big|\leq C_{k,l}
\qquad \mbox{for all $n$},
\]
where the constant, $C_{k,l}$, depends on $\left\Vert \omega_{0}\right\Vert _{\infty}$
and $\delta>0$ only, for every $\delta>0$ and $k,l\in\mathbb{N}$.
These \emph{a priori} estimates allow us to conclude that, if necessary
for a convergent subsequence, $u^{(n)}\rightarrow u$ and $\omega^{(n)}\rightarrow\omega$
(in a space with a norm including high derivatives) so that
\[
\partial_t\omega+(u\cdot\nabla)\omega-A(u)\omega-\frac{1}{2}\Delta\omega=0,
\]
where $\nabla\cdot u=0$ and $\omega=\nabla\wedge u$. Then
\begin{align*}
\nabla\wedge\big(\partial_t u+(u\cdot\nabla)u-\frac{1}{2}\Delta u\big)
 =\partial_t\omega+(u\cdot\nabla)\omega-A(u)\omega-\frac{1}{2}\Delta\omega =0,
\end{align*}
which implies that there is a scalar function $P$ such that
\[
\partial_t u+(u\cdot\nabla)u-\frac{1}{2}\Delta u=-\nabla P.
\]
Together with $\nabla\cdot u=0$, we conclude that $u$ is the
unique strong solution of the Cauchy problem \eqref{ns-p1}--\eqref{ns-p-ID}
for the Navier-Stokes equations \eqref{ns-p1}.  This completes the proof.
\end{proof}

\bigskip
\noindent
{\bf Acknowledgements}.
The research of Gui-Qiang G. Chen was supported in part by
the UK Engineering and Physical Sciences Research Council Award
EP/L015811/1, EP/V008854, and EP/V051121/1.
The research of Zhongmin Qian was supported in part by the EPSRC Centre for Doctoral Training in Mathematics of Random Systems: Analysis, Modelling and Simulation (EP/S023925/1).

\end{document}